\documentclass[reqno,11pt]{amsart}
\usepackage{amsmath, latexsym, amsfonts, stmaryrd, amssymb, amsthm, amscd, array, caption}
\usepackage{appendix}
\usepackage[makeroom]{cancel}
\usepackage[normalem]{ulem}
\usepackage[utf8]{inputenc}
\usepackage{graphics,epsf,psfrag,epsfig,color}
\usepackage[colorlinks=true, pdfstartview=FitV, linkcolor=blue, citecolor=blue, urlcolor=blue,pagebackref=false]{hyperref}

\usepackage[showlabels,sections,floats,textmath,displaymath]{}

\numberwithin{equation}{section}
\newtheorem{theorem}{Theorem}
\newtheorem{lemma}{Lemma}[section]
\newtheorem{proposition}[theorem]{Proposition}

\newtheorem{rem}{Remark}[section]

\newtheorem{hyp}{Assumption}

\renewenvironment{proof}[1][Proof]{\begin{trivlist}
\item[\hskip \labelsep {\bfseries #1}]}{\qed\end{trivlist}}

\DeclareMathOperator{\p}{\mathbb{P}}

\newcommand{\ind}{\mathbf{1}}
\renewcommand{\ge}{\geq}
\renewcommand{\le}{\leq}
\newcommand{\R}{\mathbb{R}}
\newcommand{\Z}{\mathbb{Z}}
\newcommand{\N}{\mathbb{N}}
\renewcommand{\tilde}{\widetilde}
\renewcommand{\hat}{\widehat}

\DeclareMathSymbol{\leqslant}{\mathalpha}{AMSa}{"36} 
\DeclareMathSymbol{\geqslant}{\mathalpha}{AMSa}{"3E} 
\DeclareMathSymbol{\eset}{\mathalpha}{AMSb}{"3F}     
\renewcommand{\leq}{\;\leqslant\;}                   
\renewcommand{\geq}{\;\geqslant\;}                   
\newcommand{\dd}{\,\text{\rm d}}             

\newcommand{\sumtwo}[2]{\sum_{\substack{#1 \\ #2}}} 

\newcommand{\partiald}[1]{\frac{\partial}{\partial #1}}

\newcommand{\cA}{{\ensuremath{\mathcal A}} }
\newcommand{\cB}{{\ensuremath{\mathcal B}} }

\newcommand{\cH}{{\ensuremath{\mathcal H}} }
\newcommand{\cC}{{\ensuremath{\mathcal C}} }
\newcommand{\cN}{{\ensuremath{\mathcal N}} }

\newcommand{\cZ}{{\ensuremath{\mathcal Z}} }
\newcommand{\cI}{{\ensuremath{\mathcal I}} }
\newcommand{\cJ}{{\ensuremath{\mathcal J}} }

\newcommand{\bP}{{\ensuremath{\mathbf P}} }
\newcommand{\bE}{{\ensuremath{\mathbf E}} }

\newcommand{\bbE}{{\ensuremath{\mathbb E}} }

\newcommand{\bbN}{{\ensuremath{\mathbb N}} }

\newcommand{\bbP}{{\ensuremath{\mathbb P}} }

\newcommand{\bbR}{{\ensuremath{\mathbb R}} }

\newcommand{\bbZ}{{\ensuremath{\mathbb Z}} }

\newcommand{\ga}{\alpha}
\newcommand{\gb}{\beta}
\newcommand{\gd}{\delta}
\newcommand{\gep}{\varepsilon}       

\newcommand{\gz}{\zeta}

\newcommand{\gD}{\Delta}

\newcommand{\go}{\omega}
\newcommand{\gto}{{\tilde\omega}}

\newcommand{\gl}{\lambda}

\newcommand{\gt}{\theta}
\newcommand{\gU}{\Upsilon}

\newcommand{\tf}{\mathtt{F}}
\renewcommand{\P}{{\ensuremath{\mathbf P}} }
\newcommand{\E}{{\ensuremath{\mathbf E}} }

\renewcommand{\a}{ \mathrm{a}}

\newcommand{\cpin}{\mathbf{C_{\rho}^{\rm pin}}}
\newcommand{\ccop}{\mathbf{C_{\rho}^{\rm cop}}}
\newcommand{\cop}{\mathrm{cop}}
\newcommand{\pin}{\mathrm{pin}}

\newcommand{\oJ}{\overset{\circ}{J}}

\definecolor{light-gray}{gray}{0.5}

\def\f{\textsc{f}}

\title[Critical curve for correlated pinning and copolymer models]{On the critical curves of the Pinning and Copolymer models in Correlated Gaussian environment}

\author[Q.Berger]{Quentin Berger}
\address{Department of Mathematics, KAP 108\\
University of Southern California\\
Los Angeles, CA  90089-2532 USA}
\email{qberger@usc.edu}

\author[J.Poisat]{Julien Poisat}
\address{
Mathematical Institute,
Leiden University\\
 P.O. Box 9512, NL-2300 RA Leiden\\
  The Netherlands
}
\email{poisatj@math.leidenuniv.nl}

\date{\today}


\begin{document}
\begin{abstract}

We investigate the disordered copolymer and pinning models, in the case of a correlated Gaussian environment with summable correlations, and when the return distribution of the underlying renewal process has a polynomial tail. As far as the copolymer model is concerned, we prove disorder relevance both in terms of critical points and critical exponents, in the case of non-negative correlations. When some of the correlations are negative, even the annealed model becomes non-trivial. Moreover, when the return distribution has a finite mean, we are able to compute the weak coupling limit of the critical curves for both models, with no restriction on the correlations other than summability. This generalizes the result of Berger, Caravenna, Poisat, Sun and Zygouras \cite{cf:BCPSZ} to the correlated case. Interestingly, in the copolymer model, the weak coupling limit of the critical curve turns out to be the maximum of two quantities: one generalizing the limit found in the IID case \cite{cf:BCPSZ}, the other one generalizing the so-called Monthus bound.
\end{abstract}

\maketitle

\section{Introduction}

In this paper we denote by $\bbN$ the set of positive integers, and $\bbN_0 = \bbN \cup \{0\}$.

\subsection{The copolymer and pinning models}

We briefly present here a general version of the models. For a more complete overview and physical motivations, we refer to \cite{cf:CGG,cf:G1,cf:G2,cf:dH}.
\paragraph{\it Renewal sequence.} Let $\tau=(\tau_i)_{i\geq 0}$ be a renewal process whose law is denoted by $\bP$: $\tau_0:=0$, and the
$(\tau_i-\tau_{i-1})_{i\geq 1}$'s are IID $\bbN$-valued random variables. The set $\tau=\{\tau_0,\tau_1,\ldots\}$ (with a slight abuse of notation) represents the set of contact points of the polymer with the interface, and the intervals $(\tau_{i-1},\tau_i]$ are referred as excursions of the polymer away from the interface.
We assume that the renewal process is recurrent, and that its inter-arrival distribution verifies
\begin{equation}
\label{defK}
K(n):=\bP(\tau_1=n) =\frac{\varphi(n)}{n^{1+\ga}}, \quad \text{for } n\in\bbN,
\end{equation}
where $\ga\in[0,+\infty)$, and $\varphi:(0,\infty)\to(0,\infty)$ is a slowly varying function (see \cite{cf:Bingham} for a definition).
 We also denote by $\{n\in\tau\}$ the event that there exists $k\in\N_0$ such that $\tau_k = n$ and we write $\delta_n = \ind_{\{n\in\tau\}}$.

For the copolymer model, one also has to decide whether the excursions are above or below the interface.
Take $(X_k)_{k\geq1}$ a sequence of IID Bernoulli random variables with parameter $1/2$, independent of the sequence $\tau$, whose law will be denoted by $\bP^X$:
if $X_k=1$, we identify the $k^{\rm th}$ excursion to be below the interface.
Then, we set $\gD_n :=X_k$ for all $n\in(\tau_{k-1}, \tau_k]$, so that $\Delta_n$ is the indicator function that the $n^{\rm th}$ monomer is below the interface.
The sequences $\tau$ and $X$ (with joint law $\bP\otimes \bP^X$) therefore describe the random shape of a polymer.  From now on, we write $\bP$ instead
of $\bP \otimes \bP^X$, for conciseness.\\

\paragraph{\it Disorder sequence.} Let $\go=(\go_n)_{n\geq 0}$ be a centered and unitary Gaussian stationary sequence,
whose law is denoted by $\bbP$: $\omega_n$ is the (random) charge at the $n^{\rm th}$ monomer. Its correlation function is $\rho_{n}:=\bbE[\go_0 \go_{n}]$, defined for $n\in\bbZ$, with $\rho_{-n}=\rho_n$.
The assumption that $\bbE[\go_0]=0$ and $\bbE[\go_0^2]=\rho_0=1$ is just a matter of renormalization, and do not hide anything deep.
For notational convenience, we also write $\gU:=(\rho_{ij})_{i,j\geq0}$ the covariance matrix, where $\rho_{ij}:= \bbE[\go_i \go_{j}]=\rho_{|j-i|}$, and $\gU_k$ the covariance matrix of the Gaussian vector $(\go_1,\ldots,\go_k)$.
An example of valid choice for a correlation structure is $\rho_k=(1+k)^{-a}$ for all $k\geq 0$, with $a>0$ a fixed constant, since it is convex, cf. \cite{cf:Polya}.

\begin{hyp}\rm
\label{hyp:correlations}
We assume that correlations are \emph{summable},
that is $\sum_{n\in\bbZ} |\rho_n|<+\infty$, and we define the constant $\gU_{\infty}:=\sum_{n\in\bbZ}\rho_n$. This means that $\gU$ is a bounded operator.
We also make the additional technical assumption that $\gU$ is invertible.
\end{hyp}

Note that Assumption \ref{hyp:correlations} implies that $\lim_{n\to\infty}\rho_n=0$, which entails ergodicity of $\go$, see \cite[Ch.\ 14 \S2,
Th.\ 2]{cf:Cornfeld}. For the choice $\rho_k=(1+k)^{-a}$, Assumption \ref{hyp:correlations} corresponds to having $a>1$.

\begin{rem}\rm
\label{rem:gUinfty}
The condition that $\gU$ is invertible is a bit delicate, and enables us to get uniform bounds on the eigenvalues of $\gU_k$ and $\gU_k^{-1}$. Indeed, $\gU$ is
a bounded and invertible operator on the Banach space $\ell_1(\bbN)$, so that $\gU^{-1}$ is a
bounded operator. Therefore, Assumption \ref{hyp:correlations} yields that the eigenvalues of $\gU_k$ are uniformly bounded away from $0$.
For example, one has
\begin{equation}
\label{defgUinfty}
\gU_{\infty}:=\lim_{k\to\infty} \frac{\langle \gU_k\ind_k ,\ind_k\rangle}{\langle \ind_k ,\ind_k\rangle}
=\sum_{k\in\bbZ} \rho_k>0,
\end{equation}
where $\langle \cdot\, ,\, \cdot \rangle$ denotes the usual Euclidean scalar product, and $\mathbf{1}_k$ is the vector constituted of $k$ $1$'s and then of $0$'s.

A simple case when $\gU$ is invertible is when $1=\rho_0 > 2 \sum_{k\in\bbN} |\rho_k|$: it is then diagonally dominant. More generally, one has to consider the
Laurent series associated to the Toeplitz matrix $\gU$, namely $f(\lambda) = 1+2\sum_{k\in\bbN} \rho_n \cos (\lambda n)$ (we used that $\rho_0=1$, and that
$\rho_{-n}=\rho_n$).
Then, the fundamental eigenvalue distribution theorem of Szeg\"o \cite[Ch. 5]{cf:Szego} tells that the Toeplitz operator $\gU$ is invertible if and only if $\min_{\lambda\in[0,2\pi]} f(\lambda)>0$. 
For example, if $\rho_0=1, \rho_1=1/2$ and $\rho_k=0$ for $k\geq 2$, then Assumption \ref{hyp:correlations} is not verified.
\end{rem}

\smallskip
{\bf The copolymer model.} For a fixed sequence $\go$ (quenched disorder) and parameters $\gl\in \bbR^+$, $h\in\bbR$, and $N\in \bbN$, define the following Gibbs measures
\begin{equation}\label{def:gibbs_meas_cop}
\frac{\dd \bP_{N,\gl,h}^{\go,\cop}}{\dd \bP}:= \frac{1}{Z_{N,\gl,h}^{\go,\cop}} \exp\bigg( -2\gl \sum_{n=1}^N  (\go_n+h)\gD_n\bigg)\gd_N\, ,
\end{equation}
with the partition function used to normalize the measure to a probability measure,
\begin{equation}\label{def:part_fct_cop}
Z_{N,\gl,h}^{\go,\cop}:=
 \E\bigg[\exp\bigg( -2\gl \sum_{n=1}^N  (\go_n+h)\gD_n\bigg) \gd_N\bigg].
\end{equation}
This measure corresponds to giving a penalty/reward (depending on its sign) $\go_n+h$ if the $n^{\rm th}$ monomer is below the interface.

One then defines the free energy of the system.
\begin{proposition}[cf.\ \cite{cf:G1}, Theorem 4.6]
The following limit exists and is constant $\bbP$-a.s.
\begin{equation}
\tf^{\cop}(\gl,h):= \lim_{N\to\infty}\frac1N \log Z_{N,\gl,h}^{\go,\cop} = \lim_{N\to\infty}\frac1N \bbE \log Z_{N,\gl,h}^{\go,\cop}.
\end{equation}
It is called the \emph{quenched} free energy of the system. The map $h\mapsto \tf^{\cop}(\gl,h)$ is non-negative, non-increasing and convex. There exists a
\emph{quenched} critical point $h_c^{\cop}(\gl):= \inf\{h\, ;\, \tf^{\cop}(\gl,h)=0\}$, such that $\tf^{\cop}(\gl,h)>0$ if and only if $h<h_c^{\cop}(\gl)$.
\end{proposition}
Since it does not change the value of the free energy, we work in some places with the free version of the model, which is obtained by removing the
constraint $\{N\in \tau\}$ in (\ref{def:gibbs_meas_cop}) and (\ref{def:part_fct_cop}).

A straightforward computation shows that $\partial_h \tf^{\cop}(\gl,h)$ is the limiting fraction of monomers below the interface under the measure
$\bP_{N,\gl,h}^{\go,\cop}$ (and $\partial_h \tf^{\cop}(\gl,h)$ exists for $h<h^\cop_c(\gl)$, see \cite{cf:GT2}; differentiability at the critical point is a
consequence of the smoothing inequality, see Proposition \ref{pr:smoothcop}).
Therefore, the critical point $h^\cop_c(\gl)$ marks the transition between a delocalized phase ($\partial_h \tf^{\cop}=0$), where most of the monomers lie
above the interface, and a localized phase ($\partial_{h} \tf^{\cop}>0$), where the polymer sticks around the interface.

One also introduces the annealed counterpart of the model, to be compared with the quenched one.
The annealed partition function is
\begin{equation}
\label{eq:annCop}
\bbE\big[ Z_{N,\gl,h}^{\go,\cop} \big]=
 \E\bigg[\exp\bigg( -2\gl h\sum_{n=1}^N  \gD_n  + 2\gl^2 \sum_{n,m=1}^N
      \rho_{nm} \gD_n\gD_m\bigg)\gd_N \bigg] \, ,
\end{equation}
and the annealed free energy is
\begin{equation}
\tf^{\cop}_\a(\gl,h):=\lim_{N\to\infty} \frac{1}{N} \log \bbE \big[Z_{N,\gl,h}^{\go,\cop}\big] \geq 0.
\end{equation}
The existence of this limit
(which is a non-trivial fact if correlations can be negative) can be proved using Hammersley's approximate subadditive lemma (Theorem~2 in
\cite{cf:H}). We refer to Proposition 2.1 in \cite{cf:P1} for a detailed proof in the context of the correlated pinning model.
The annealed critical point is then defined as $h_\a^{\cop}(\gl):=\inf\{h\, ; \, \tf^{\cop}_\a(\gl,h)=0\}$.
Moreover, a simple application of Jensen's inequality gives 
\begin{equation}
\label{eq:cop_annbound}
\tf^{\cop}(\gl,h)\leq \tf^{\cop}_\a(\gl,h),\quad \mbox{so that}\quad h_c^{\cop}(\gl)\leq
h_\a^{\cop}(\gl).
\end{equation}

\smallskip
{\bf The pinning model.}
The pinning model follows similar definitions, that we state very briefly. Although the parametrization we use is a bit different than that of the copolymer model, it is conform to the existing literature.

For a fixed sequence $\go$ (quenched disorder) and parameters $\gb\in\bbR^+$, $h\in\bbR$, one defines the Gibbs measures
\begin{equation}
\frac{\dd \bP_{N,\gb,h}^{\go,\pin}}{\dd \bP}:= \frac{1}{Z_{N,\gb,h}^{\go,\pin}} \exp\bigg(  \sum_{n=1}^N  (\gb\go_n+h)\gd_n\bigg)\gd_N \, ,
\end{equation}
where the partition function is
\begin{equation}
Z_{N,\gb,h}^{\go,\pin}:=
 \E\bigg[\exp\bigg(  \sum_{n=1}^N  (\gb\go_n+h)\gd_n\bigg)\gd_N  \bigg],
\end{equation}
and which corresponds to giving a reward/penalty $\gb\go_n+h$ if the polymer touches the defect line at site $n$.
The \emph{quenched} free energy is defined as the $\bbP$-a.s limit
\begin{equation}
\tf^{\pin}(\gb,h):=\lim_{N\to\infty} \frac{1}{N} \log 
Z_{N,\gb,h}^{\go,\pin}=\lim_{N\to\infty} \frac{1}{N} \bbE\log  Z_{N,\gb,h}^{\go,\pin}\geq 0,
\end{equation}
and the \emph{quenched} critical point
$h_c^{\pin}(\gb):=\sup\{h\, ;\, \tf^{\pin}(\gb,h)=0\}$ separates a delocalized phase ($h< h_c^{\pin}(\gb)$, $\tf^{\pin}(\gb,h)=0$) and a localized phase
($h>h_c^{\pin}(\gb)$, $\tf^{\pin}(\gb,h)>0$).

One also defines the annealed free energy $\tf_\a^{\pin}(\gb,h):=\lim_{N\to\infty} \frac{1}{N} \log \bbE \big[ Z_{N,\gb,h}^{\go,\pin} \big]$, and the annealed critical
point $h_\a^{\pin}(\gb):=\sup\{h\, ;\, \tf_\a^{\pin}(\gb,h)=0\}$. Analogously to the copolymer model, $h_c^{\pin}(\gl)\geq h_\a^{\pin}(\gl)$.

\begin{rem}\rm
The choice of a Gaussian structure for the disorder $\go$ is very natural, and in addition, is essential. In the Gaussian case, the two-point correlation
function is enough to describe the whole correlation structure and to compute explicitly exponential moments. In particular, it allows us to get
an explicit annealed model, see \eqref{eq:annCop}, which is a central tool in this work. We also stress that when correlations are not summable, the
annealed model is degenerate, and the quenched free energy is always positive. We refer to \cite{cf:B13,cf:B13bis} for an explanation of this so-called
\emph{infinite disorder} phenomenon in the pinning model (the copolymer model follows the same features). This degenerate behavior is actually due to more
complex properties of the correlation structure (cf.\ Definition 1.5 in \cite{cf:B13bis}), and avoiding this issue is another reason to restrict to the Gaussian
case.
\end{rem}

\subsection{The main results}
A question of importance in these two models is that of the influence of disorder: one compares the characteristics of the quenched and annealed models, to see if they differ. In the copolymer and pinning models, this question is addressed both in terms of critical points and in terms of the order of the phase transition (that is, the lack of regularity of the free energy at the critical point).
When disorder is relevant, the question of the weak-coupling asymptotic behavior of the quenched critical point is also investigated.
We present here results on disorder relevance for both pinning and copolymer models. For each model, we give a short overview of the existing literature, and expose our results. 

\subsubsection{The Copolymer Model}
So far, the copolymer model has been studied only in the case of an IID sequence $\go$. In that case, disorder has been shown to be relevant for all $\ga>0$.
Indeed, the annealed phase transition is trivially of order $1$ whereas the quenched phase transition is, by the smoothing inequality \cite{cf:GT}, at least of order $2$. Moreover, it has been shown in \cite{cf:BGLT, cf:T1} that $h_c^{\cop}(\gl)< h_\a^{\cop}(\gl)$ for all $\gl>0$ . Much attention
has then been given to the weak coupling behavior ($\gl\downarrow 0$) of the critical curve.

In \cite{cf:BdH}, Bolthausen and den Hollander focused on the special case where the underlying renewal is given by the return times of the simple symmetric
random walk on $\Z$ (where $\ga = 1/2$), and where the $\omega_n$'s are IID, $\{\pm 1\}$-valued and symmetric.
They proved the existence of a \emph{continuum copolymer model}, in which the random walk is replaced by a Brownian motion and the disorder sequence $\omega$ by
white noise, as a scaling limit of the discrete model. They showed in particular that the slope of the critical curve $\lim_{\gl\downarrow 0}
h_c^{\cop}(\gl)/\gl$ exists, and is equal to the critical point of the continuum model.
This result has been extended by Caravenna and Giacomin~\cite{cf:CG} for the general class of copolymer models that we consider in this paper, with
$\ga\in(0,1)$: the slope of the critical curve exists, and is the critical point of a suitable $\alpha$-continuum copolymer model. In particular, the slope is
shown to be a \emph{universal quantity}, depending only on $\alpha$, and not on the fine details of the renewal process $\tau$ or on the law of the disorder
$\omega$. We then define, at least for $\ga\in(0,1)$,
\begin{equation} \label{eq:slope}
	m_\alpha \,:=\, \lim_{\gl\downarrow 0} \frac{h_c^{{\rm cop}}(\lambda)}{\gl} \,.
\end{equation}

The value of $m_\alpha$ has been the subject of many investigations and debates the past few years. In \cite{cf:M}, Monthus conjectured that $m_{1/2} = 2/3$, and a generalization of her non-rigorous renormalization argument predicts $m_\alpha = 1/(1+\ga)$. Bodineau and Giacomin \cite{cf:BG} proved the lower bound $m_\alpha \ge 1/(\ga+1)$, for every $\alpha \ge 0$. Monthus' conjecture was ruled out first by Bodineau, Giacomin, Lacoin and Toninelli \cite[Theorem 2.9]{cf:BGLT} for $\ga\geq 0.801$, and more recently by Bolthausen, den Hollander and Opoku \cite{cf:BdHO} for $\ga>0$. We also refer to \cite{cf:BGLT} for earlier, partial results, and \cite{cf:CGG} for a numerical study in the case $\ga=1/2$.

The case $\ga>1$ was not considered until recently, in particular
because no non-trivial continuum model is expected to exist, due to the finite mean of the excursions of the renewal process. However, it was proved recently in
\cite{cf:BCPSZ} that the slope $m_{\ga}$ exists also for $\ga>1$, and the exact value was found to be $m_{\ga}=\frac{2+\ga}{2(1+\ga)}$. This answered a conjecture of Bolthausen, den Hollander and Opoku \cite{cf:BdHO}, who had already proved the matching lower bound for the slope.

\medskip
We now turn to the correlated version of the copolymer model. The annealed model already presents some surprising features. When correlations are non-negative, it is still a trivial model to study, but the case with negative correlations is challenging, and more investigation would be needed (see Remark \ref{rem:negativecopolann}). Propositions \ref{prop:copann} to \ref{pr:smoothcop} are valid for all $\lambda>0$ whereas Theorems \ref{thm:copolrel} and \ref{thm:gapcop} deal with the weak-coupling regime ($\lambda\downarrow 0$).

\begin{proposition}
\label{prop:copann}
 If correlations are non-negative, then for any $\lambda\in \bbR^+$ and $h\in\bbR$, the annealed free energy is
\begin{equation}
\tf_\a^{\cop}(\gl,h) = 2\lambda \big( \gU_{\infty}\lambda -h\big)_+ \, ,
\end{equation}
where we used the notation $x_+ = \max(x,0)$.
Therefore, the annealed critical point is
\begin{equation}\label{eq:copanncc}
h_{\a}^{\cop}(\gl) = \gl \gU_{\infty},
\end{equation}
and the annealed phase transition is of order $1$.
\end{proposition}

\begin{proof}
From \eqref{eq:annCop}, one has the easy lower bound, 
\begin{equation}
\bbE\big[ Z_{N,\gl,h}^{\go,\cop} \big] \geq \P(\tau_1=n, \gD_1=1)\exp \Big(2\gl N \Big( -h+\frac{\gl}{N} \sum_{n,m=1}^N
      \rho_{nm} \Big) \Big),
\end{equation}
as well as $\bbE\big[ Z_{N,\gl,h}^{\go,\cop} \big] \geq \P(\tau_1=n, \gD_1=0)$, which directly gives that
\begin{equation}\label{eq:lowerbound_Facop}
\tf_\a^{\cop}(\gl,h)\geq 2\gl \big(\gU_{\infty}\gl -h \big)_+, 
\end{equation} using in particular the assumption on the renewal \eqref{defK}. Note that this does not require the non-negativity assumption. For the upper bound, one uses that for $n\in\bbN$, $\sum_{m\geq1} \rho_{nm} \gD_m \leq \gU_\infty,$ which is valid only for non-negative correlations.
\end{proof}

Our next result is a general bound on the quenched critical curve, which is the analogous to that of Bodineau and Giacomin \cite{cf:BG} in the correlated case, with no restriction on the sign of correlations.
\begin{proposition}
\label{prop:BG}
For $\ga\geq 0$ and any $\gl>0$,
\begin{equation}
 \frac{\gU_{\infty}}{1+\ga} \leq \frac{h_c^{\cop}(\gl)}{\gl} \leq \frac{h_\a^{\cop}(\gl)}{\gl},
\end{equation}
and we stress that $\frac{h_a^{\cop}(\gl)}{\gl} \geq \gU_{\infty}$ (see \eqref{eq:lowerbound_Facop}), with equality when correlations are non-negative.
\end{proposition}
The upper bound is standard and has been already pointed out in \eqref{eq:cop_annbound}. The lower bound is the so-called Monthus bound, adapted to the correlated case. Its proof is postponed to Section \ref{sec:lowcopol}. Note that if $\alpha =0$ and the correlations are non-negative we get $h_c^\cop(\gl) =h_\a^{\cop}(\gl)= \Upsilon_\infty \gl$, for all $\lambda>0$.

Another general result on the quenched copolymer model is the so-called smoothing inequality \cite{cf:GT}, which is also valid in the correlated case, without any restriction on the sign of the correlations.
\begin{proposition}
\label{pr:smoothcop}
For every $\gl\geq 0$ and $ \gd\geq 0$, one has
\begin{equation}
0\leq \tf(\gl,h_c^{\cop}(\gl)-\gd) \leq \frac{1+\alpha}{2\gU_{\infty}}\gd^2.
\end{equation}
\end{proposition}
It is proved in the same way as in the pinning model, see \cite[Sec. 4]{cf:B13}, and a brief sketch of the proof is given in Section \ref{sec:rarestretch}. Together with Proposition \ref{prop:copann}, this result also shows disorder relevance for all $\ga\geq0$ in terms of critical exponents, in the case of non-negative correlations, since the annealed phase transition is then known to be of order $1$.

We are also able to show disorder relevance in terms of critical points, with the following result, similar to \cite[Theorem 2.1]{cf:T1} in the IID case.
\begin{theorem}
\label{thm:copolrel}
If correlations are non-negative, then for all $\ga>0$, there exists $\theta(\ga)<1$ such that 
\begin{equation}
\limsup_{\gl\downarrow 0} \frac{h_c^{\cop}(\gl)}{\gl} \leq \theta(\ga) \gU_{\infty}.
\end{equation}
\end{theorem}
Since $h_\a^{\cop}(\gl) = \gU_{\infty}\gl$ when correlations are non-negative, this proves in particular that $h_c^{\cop}(\gl)<h_\a^{\cop}(\gl)$ for $\gl$ small enough, that is disorder relevance in terms of critical points for all $\ga >0$.

As far as the slope of the critical curve is concerned, we strongly believe that the proof of \cite{cf:CG} is still valid with correlated disorder, and that the
slope exists for $\alpha\in (0,1)$. Reproducing Step 2 in \cite[Section 3.2.]{cf:CG}, one would presumably end up with a continuum $\alpha$-copolymer where the disorder is given by a standard Brownian motion multiplied by a corrected variance $\gU_\infty$. We therefore suspect that for $\ga\in (0,1)$, the slope for correlated disorder is the slope for IID disorder multiplied by a factor $\gU_\infty$.
\smallskip

We focus now on the case $\mu:=\bE[\tau_1]<+\infty$, for which we manage to identify the weak-coupling limit of the critical curve, without any restriction on the signs of the correlations. This result is the analogous to \cite[Theorem 1.4]{cf:BCPSZ} in the correlated case.

\begin{theorem}\label{thm:gapcop}
For the correlated copolymer model with $\mu=\bE[\tau_1]<+\infty$,
\begin{equation}\label{eq:limcop}
\lim_{\gl\downarrow 0} \frac{h_c^{\cop}(\gl)}{\gl} = \max\left\{\frac{\gU_{\infty}}{1+\alpha}\, ;\, \frac12 \frac{\gU_{\infty}}{1+\ga} + \frac12 \ccop\right\}
\end{equation}
with
\begin{equation}
\label{defccop}
\ccop:= \bE\bigg[ \frac{1}{\mu}\sum_{n,m=1}^{\tau_1} \rho_{nm} \bigg] =\sum_{n\in\Z} \rho_n \kappa_n,
\end{equation}
where we defined, for $n\in\Z$,
$\kappa_n = \frac{1}{\mu} \sum_{k\geq |n|}\P(\tau_1\geq k  +1) \in [0,1]$. We stress that the $\kappa_n$'s have a probabilistic interpretation in terms of the tail of the invariant measure of the backward recurrence time, see Appendix \ref{app1}.
\end{theorem}
Let us make a few remarks about this result:

\smallskip\noindent
{\bf 1.} One recovers that $m_{\ga}=\frac{2+\ga}{2(1+\ga)}$ in the IID case, because then $\gU_{\infty}=\ccop=1$. It slightly improves Theorem 1.4 in \cite{cf:BCPSZ}, replacing the condition $\ga > 1$ by $\mu < +\infty$. This comes from an improvement in \eqref{eq:bound2} and \eqref{eq:sumalpha=1}, where we use that $\ga = 1$ and $\mu < +\infty$ imply that $\varphi(n) \downarrow 0$ as $n\uparrow\infty$, see \cite[Proposition 1.5.9b]{cf:Bingham}.

\smallskip\noindent
{\bf 2.} The slope is the maximum of two terms: the first term is the generalization of the Monthus bound to the correlated case, whereas the second term is the generalization of the slope found in the IID case \cite[Theorem 1.4]{cf:BCPSZ}.

\smallskip\noindent
{\bf 3.} The Monthus bound $\frac{\gU_{\infty}}{1+\alpha}$ was ruled out in the IID case, except in degenerate examples ($\ga=0$, or the ``reduced" wetting model, see \cite[Theorem 3.4]{cf:T08}).
In the correlated case, it turns out to be the correct limit in some cases, namely when
\begin{equation}
 \frac{\gU_{\infty}}{1+\alpha} - \left(  \frac12 \frac{\gU_{\infty}}{1+\ga} + \frac12 \ccop \right) = \frac12 \left( \frac{\gU_{\infty}}{1+\alpha} - \ccop \right) >0,
\end{equation}
a condition that we can rewrite as:
\begin{equation}\label{eq:monthus_crit}
 \frac{1}{\gU_{\infty}}\sum_{n\in\Z} \rho_n \kappa_n < \frac{1}{1+\alpha} \in (0,1/2].
\end{equation}
At least when the correlations are non-negative, the left-hand side of (\ref{eq:monthus_crit}) can be interpreted as a probability: let $U$ and $V$ be independent random variables with distribution
\begin{align}
 &\forall n\in\bbN_0,\quad \bP(U=n) = \frac{1}{\mu} \bP(\tau_1 \geq n+1), \quad \big( \text{so that } \kappa_n = \bP(U\geq |n|) \big)\\
 &\forall n\in\Z,\quad \bP(V=n) = \rho_n /\gU_{\infty},
\end{align}
then (\ref{eq:monthus_crit}) is equivalent to 
$ \bP(U \geq |V|) < 1/(1+\ga).$
Besides,
\begin{equation}
 \bP(U \geq |V|) \leq \bP(U \geq 1) + \bP(|V| = 0) = 1 - 1/\mu + \gU_\infty^{-1},
\end{equation}
and it can be made arbitrarily small by choosing $K(1)$ close to $1$ and $\gU_\infty$ large enough, so that (\ref{eq:monthus_crit}) holds.

\smallskip \noindent
{\bf 4.} In the case of non-negative correlations, $\ccop \leq \gU_{\infty}$, so that $\lim_{\gl\downarrow 0} \frac{h_c^{\cop}(\gl)}{\gl} \leq \frac{2+\ga}{2(1+\ga)} \, \gU_{\infty}$ (with possibly a strict inequality, as mentioned above). However, with negative correlations, it might be the case that $\ccop>\gU_{\infty}$: take for instance $\rho_0=1$, $\rho_1<0$ and $\rho_k=0$ for $k\geq 2$, so that  $\ccop=1+2\rho_1 (1-1/\mu)>1+2\rho_1 = \gU_{\infty}$ (since $\kappa_0=1$, $\kappa_1=1-1/\mu$). One then would have $\lim_{\gl\downarrow 0} \frac{h_c^{\cop}(\gl)}{\gl} > \frac{2+\ga}{2(1+\ga)} \, \gU_{\infty}$.


\begin{rem} \textit{Annealed system with negative correlations.}\rm
\label{rem:negativecopolann}
The lower bound in \eqref{eq:lowerbound_Facop} comes from the trajectories that makes one large excursion below the interface. This strategy is optimal when the correlations are non-negative, because returning to the interface would only result in a loss of some positive $\rho_{mn}$. Other strategies may actually give better bounds when correlations are allowed to be negative.
For example, using Jensen's inequality on the free annealed partition function, one gets
\begin{equation}
\label{eq:lw.jensenstrategy}
\bbE\big[ Z_{N,\gl,h}^{\go,\cop,{\rm free}} \big]\geq  \exp\Big( -\gl hN +2\gl^2\sum_{n,m=1}^N  \rho_{nm}\bE[\gD_n \gD_m]  \Big).
\end{equation}
If $\mu=\bE[\tau_1]<+\infty$, Lemma \ref{lem:convergenceapp} gives that $\tf_\a^{\cop}(\gl,h) \geq  \lambda \big[ (\frac12 \gU_{\infty}+\frac12 \ccop )\lambda -h\big]_+ $, and 
\begin{equation}
\label{eq:lowerbound_bitbetter}
h_\a^{\cop} (\gl)\geq \gl  \Big(\gU_{\infty} +\frac12(\ccop-\gU_{\infty})_+ \Big).
\end{equation}
One therefore gets that $h_c(\gl)> \gU_{\infty} \gl$ if $\ccop>\gU_{\infty}$ (which can happen, see point {\bf 4} above). 
The strategy highlighted in \eqref{eq:lw.jensenstrategy} is for the renewal to come back to the origin in a typical manner (and lose some negative $\rho_{mn}$'s), and simply use the fact that $2 \times \bE\big[\sum_{n,m=1}^N \rho_{mn} \gD_n \gD_m \big]$ is strictly greater than $ \gU_{\infty} N$ (this doesn't apply when $\alpha<1$, because on average, you don't return enough to the origin, cf. Remark \ref{rem:muinfty}).
Further investigation needs to be carried out to understand  the annealed phase transition in presence of negative correlations.
\end{rem}

\subsubsection{The Pinning Model}
For the pinning model with an IID sequence $\go$, the question of the influence of disorder has been extensively studied, and the so-called Harris prediction \cite{cf:Harris} has been mathematically settled. First, it is known that the annealed phase transition is of order $\max(1,1/\ga)$, and $h_\a^{\pin}(\gb)=-\log \bbE[e^{\gb \go_1}]{\sim} {-\gb^2/2}$, as ${\gb\downarrow 0}$.
Disorder has been proved to be irrelevant if $\ga<1/2$ or if $\ga=1/2$ and $\sum n^{-1}\varphi(n)^{-2}$ is finite, in which case, for $\beta$ small enough, $h_c^{\pin}(\gb)=h_\a^{\pin}(\gb)$ and the critical behavior of the quenched free energy is that of the annealed one
\cite{cf:A2,cf:AZ2,cf:CdH,cf:L,cf:T2}.
The relevant disorder counterpart of these results has also been proved.
In \cite{cf:GT}, the disordered phase
transition is shown to be of order at least $2$, proving disorder relevance for $\ga>1/2$. In terms of critical points, it has been proved that when $\ga>1/2$ or
$\ga=1/2$ and $\varphi(n)=o((\log n)^{1/2-\eta})$ for some $\eta>0$, then $h_c^{\pin}(\gb)>h_\a^{\pin}(\gb)$, see \cite{cf:AZ,cf:CdH,cf:DGLT,cf:GLT1,cf:GLT2}. Moreover, the weak-coupling asymptotic of $h_c^{\pin}(\gb)$ has been computed in \cite{cf:BCPSZ}, when $\ga>1$, $\mu=\bE[\tau_1]<+\infty$:
\begin{equation}
\label{eq:resultBCPSZ1}
\lim_{\gb\downarrow 0}\frac{ h_c^{\pin}(\gb)- h_\a^{\pin}(\gb)}{\gb^2}=\frac{1}{2\mu} \frac{\ga}{1+\ga}.
\end{equation}
For the case $\ga \in (1/2,1)$, we refer to Conjecture 3.5 in the recent paper \cite{cf:CSZ14}, which also provides a new perspective in the study of disorder relevance (and beyond pinning models).

In the correlated case, a few steps have been made towards the same type of criterion.
First, the annealed model is not trivial, and is still not completely solved: in particular, although a spectral characterization of the annealed critical curve is given in \cite{cf:P1}, there is no explicit formula as in the IID case.
\begin{proposition}[cf.\ \cite{cf:B13,cf:P1}]
\label{prop:pinann}
The following limit holds:
\begin{equation}
\label{eq:cpin}
\lim_{\gb \downarrow 0} \frac{h_{\a}^{\pin}(\gb)}{\gb^2} = -\frac12 \cpin,
\quad \text{with }\ \cpin:= \sum_{n\in \bbZ} \rho_n \P(|n|\in\tau).
\end{equation}
Moreover, if $\sum_{n\in\bbN} n |\rho_n|$ is finite then for all $\beta\geq0$, there exist $c_\beta,C_\gb \in (0,\infty)$ such that for $u\geq 0$,
\begin{equation}
\label{eq:pin.sumcorr.order}
c_\gb\, u^{\max(1,1/\ga)} \leq \tf_\a(\gb,h_\a^\pin(\gb)+u) \leq C_\gb\, u^{\max(1,1/\ga)}.
\end{equation}
\end{proposition}
The smoothing inequality \cite{cf:GT} has also been extended to the correlated case:
\begin{proposition}[See \cite{cf:B13}, Theorem 2.3]
\label{pr:smoothpin}
For every $\gb>0$ and $\gd\geq0$,
\begin{equation}
\label{eq:pr:smoothpin}
0\leq \tf(\gb,h_c^{\pin}(\gb)+\gd) \leq \frac{1+\alpha}{2\gU_{\infty}} \frac{\gd^2}{\gb^2}.
\end{equation}
\end{proposition}
Therefore, from \eqref{eq:pin.sumcorr.order} and \eqref{eq:pr:smoothpin}, one deduces disorder relevance (in terms of critical exponents) for $\ga>1/2$ and $\sum_{n\in\bbN} n |\rho_n|<\infty$.

Our main result concerning the correlated pinning model is that we identify the small coupling asymptotic of the quenched critical point when $\mu<+\infty$, in analogy with \eqref{eq:resultBCPSZ1}.
\begin{theorem}
\label{thm:gappin}
For the correlated pinning model with $\mu=\bE[\tau_1]<+\infty$,
\begin{equation}
\lim_{\gb\downarrow 0} \frac{h_{c}^{\pin}(\gb)- h_{\a}^{\pin}(\gb)}{\gb^2} = \frac{\gU_{\infty}}{2\mu} \frac{\ga}{1+\ga}. 
\end{equation}
\end{theorem}
Notice that the asymptotics given in Theorem~\ref{thm:gappin} and that of \eqref{eq:resultBCPSZ1} only differ through the multiplicative constant $\gU_{\infty}$. In particular, one recovers \eqref{eq:resultBCPSZ1} in the IID case, where $\gU_{\infty}=1$. Theorem \ref{thm:gappin} proves disorder relevance in terms of critical points if $\mu<+\infty$ (in particular if $\ga>1$), under Assumption \ref{hyp:correlations}. 

\subsection{Outline of the proofs}
We mostly focus on the copolymer model, which has not been investigated so far in the correlated case, and we give a detailed proof only in that case. The pinning model essentially follows the same scheme.

In Section \ref{sec:lowcopol}, we deal with lower bounds on the free energy, leading to lower bounds on the critical curve. The basic ingredient is a rare-stretch strategy, already widely used in the literature. This approach was initiated in \cite{cf:BG,cf:GT} in the context of the pinning and copolymer models. We briefly recall how to use this technique and explain how it is modified by the presence of correlations. It is then applied to get the smoothing inequality of Proposition \ref{pr:smoothcop}, the Monthus bound of Proposition \ref{prop:BG}, and thanks to the additional estimate in Lemma \ref{lem:linearbound} (analogous to \cite[Lemma 5.1]{cf:BCPSZ} in the correlated case), we get the second lower bound of Theorem \ref{thm:gapcop}, that is $\liminf_{\gl\downarrow 0} \frac{h_c(\gl)}{\gl} \geq \frac12 \frac{\gU_{\infty}}{1+\ga} + \frac12 \ccop$.

In Section \ref{sec:uppercopol}, we deal with the upper bound in Theorem \ref{thm:gapcop}. We employ a standard technique which was introduced by \cite{cf:DGLT} in the context of the pinning model, and developed in \cite{cf:GLT1,cf:GLT2,cf:BCPSZ}: it is the so-called fractional moment method, combined with a coarse-graining argument. However, the adaptation of this technique to the correlated case requires considerable work: Lemma \ref{lem:finitefraccop} allows us to control the fractional moment of the partition function on length scale $1/\gl^2$, whereas Lemma \ref{lem:decouple} helps us to control the correlation terms in the coarse-graining argument, which is a necessary step to glue the finite-size estimates together.

Section \ref{sec:adaptpin} adapts the proofs to the pinning model. We begin with Lemma \ref{lem:lwpin} (analogous to \cite[Lemma 3.1]{cf:BCPSZ}), which is the pinning model counterpart of Lemma \ref{lem:linearbound}. Our proof relies on Gaussian interpolation techniques. Combining this lemma with the smoothing inequality of Proposition \ref{pr:smoothpin} is the key to obtain the upper bound in Theorem \ref{thm:gappin}, that is $\limsup_{\gl\downarrow 0} \frac{h_c^{\pin}(\gl) - h_\a^{\pin}(\gl)}{\gl} \leq \frac{\gU_{\infty}}{2\mu} \frac{\ga}{1+\ga}$. The lower bound for the pinning model follows the same fractional moment/coarse-graining scheme as for the copolymer model, and one mostly needs to adapt the finite-size estimates of the fractional moment, see Lemma \ref{lem:finitefractpin}. Together with minor changes in the coarse-graining procedure, this yields the right lower bound in Theorem \ref{thm:gappin}.

\section{On the copolymer model: lower bound}
\label{sec:lowcopol}

\subsection{Rare-stretch strategy}
\label{sec:rarestretch}
To get a lower bound on the free energy, it is enough to highlight particular trajectories, and show that these contribute enough to the free energy to make it positive.
To that purpose, we use a rare stretch strategy: we only consider trajectories which stay above the interface (where there is no interaction) until they reach favorable regions of the interface and localize. We now describe how to implement this idea for the sake of completeness, but we omit details, since the procedure is standard (since \cite{cf:GT}), and was already used in a correlated framework in \cite{cf:B13bis}.

Let us fix $L$ a large constant integer, and divide the system in blocks of length $L$, denoted by $(B_i)_{i\in\bbN}$, $B_i:=\{ (i-1) L+1,\ldots,i L \}$.
Then, one restricts the trajectories to visit only blocks $B_i$ for which $(\go_j)_{j\in B_i}\in\cA$, where $\cA$ is an event corresponding to a ``good" property of $\go$ on $B_i$. There are many ways to define $\cA$, but for our purpose, we consider
\begin{equation}
\cA=\Big\{ (\go_j)_{1\leq j\leq L}\ ;\ \frac{1}{ L} \log Z_{ L,\gl,h}^{\go,\cop} \geq \tf^{\cop}(\gl, h-a) -\gep \Big\},
\end{equation}
for some $a\in\bbR$, and $\gep>0$ fixed but meant to be small. Let us denote by $(i_k)_{k\in \bbN}$ the (random) indices of the good blocks ($(\go_j)_{j\in B_{i_k}}\in\cA$). Restricting the partition function to trajectories which only visit the blocks $(B_{i_k})_{k\in\bbN}$ gives
\begin{equation}
 Z_{i_n  L, \gl,h}^{\go,\cop} \geq\prod_{k=1}^n K\big( (i_k-i_{k-1}-1) L \big)  Z_{B_{i_k}}
\end{equation}
where $Z_{B_{i_k}}:= Z_{ L,\gl,h}^{\theta^{ i_k  L}\go,\cop}$ and $\theta$ is the shift operator ($\theta \go = (\go_{n+1})_{n\in \bbN} $).
We may now choose $ L$ large enough such that $(1+\ga+\gep^2) \frac{1}{ L} \log  L \leq \gep$, and by \eqref{defK}, write
\begin{multline}
\frac{1}{i_n  L } \log Z_{i_n  L, \gl,h}^{\go,\cop} \geq \frac{n}{i_n} \frac{1}{n} \sum_{k=1}^n \bigg(-(1+\ga+\gep^2) \frac{1}{ L} \log\big( (i_k-i_{k-1}) L \big)  + \frac{1}{ L}  \log Z_{B_{i_k}} \bigg)\\
\geq \frac{n}{i_n} \frac{1}{n} \sum_{k=1}^n \bigg( -(1+\ga+\gep^2) \frac{1}{ L} \log\big( i_k-i_{k-1} \big)  + \tf^{\cop}(\gl,h-a)-2\gep \bigg),
\end{multline}
where we used the definition of the event $\cA$ to estimate $\frac{1}{ L} \log Z_{B_{i_k}}$. Taking the limit $n\to\infty$, and using (twice) Birkhoff's ergodic Theorem, one gets
\begin{equation}
\tf^{\cop}(\gl,h) = \liminf_{N\to\infty} \frac{1}{N} \log Z_{N,\gl,h}^{\go,\cop}\geq \frac{1}{\bbE[i_1]} \big(-(1+\ga +\gep^2) \frac{1}{ L} \bbE \log i_1  + \tf^{\cop}(\gl,h-a)-2\gep \big).
\end{equation}
Since $\bbE \log i_1 \leq  \log \bbE i_1$ and $ \bbE i_1= \bbP(\cA)^{-1}$, one is left to estimate $\bbP(\cA)$. To that end, one uses a change of measure argument. Let us consider $\tilde \bbP_{ L}$ the measure consisting in shifting $(\go_1,\ldots,\go_{ L})$ by $-a$, so that under $\tilde \bbP_{ L}$, the event $\cA$ is typical. Indeed, shifting $\go$ by $-a$ corresponds to a shift of the parameter $h$ by $-a$ in the partition function. Therefore, when $ L$ goes to infinity, $\tilde\bbP_{ L} (\cA)= 1+o(1)$. Then, we use the standard entropy inequality 
\begin{equation}
\bbP(\cA) \geq \tilde \bbP_{ L} (\cA) \exp\left( -\tilde\bbP_{ L}(\cA)^{-1} ({\rm H}(\tilde\bbP_{ L}|\bbP)+e^{-1}) \right),
\end{equation}
 where ${\rm H}(\tilde\bbP_{ L}|\bbP)$ denotes the relative entropy of $\tilde\bbP_{ L}$ with respect to $\bbP$, and verifies
\begin{equation}
{\rm H}(\tilde\bbP_{ L}|\bbP) = \frac{a^2}{2} \langle \gU_{ L}^{-1} \ind_{ L}\, , \ind_{ L} \rangle = (1+o(1)) \frac{a^2}{2\gU_{\infty}}\,  L ,
\end{equation}
where the last equality comes from Lemma A.1 in \cite{cf:B13} (and uses Assumption \ref{hyp:correlations}). In the end, we get
\begin{equation}
\frac{1}{ L}\log\bbP(\cA) = -\frac{a^2}{2\gU_{\infty}}(1+o(1)).
\end{equation}
Then, for a fixed $\gep>0$ chosen small enough, the following lower bound holds for $\gl, a\in \bbR^+$, provided that $L=L(\gep)$ is large enough,
\begin{equation}
\label{eq:rarestretch}
\tf^{\cop}(\gl, h) \geq \bbP(\cA)\bigg( \tf(\gl, h-a ) - (1+\ga) \frac{a^2}{2\gU_{\infty} } - 3 \gep \bigg),
\end{equation}
It is then straightforward to get the smoothing inequality, Proposition~\ref{pr:smoothcop}. Indeed, evaluating \eqref{eq:rarestretch} at $h=h_c^{\cop}(\gl)$, we get $\tf(\gl, h_c^{\cop}(\gl))=0$ in the left-hand side, and therefore,
\begin{equation}
\tf^{\cop}(\gl, h_c^{\cop}(\gl)- a ) - (1+\ga) \frac{a^2}{2\gU_{\infty} } - 3 \gep\leq 0.
\end{equation}
The result follows by letting $\gep$ go to $0$.
\qed

\subsection{Application: lower bounds on the critical point}
\label{sec:applicrarestrecth}
Non-trivial bounds on the critical point follow from \eqref{eq:rarestretch}. In particular, it is straightforward to get
\begin{equation}
\label{eq:borneinfh}
\sup_{a\in\bbR}\Big\{ \tf^{\cop}(\gl, h - a ) - (1+\ga) \frac{a^2}{2\gU_{\infty}} \Big\}  >0 \quad \Rightarrow\quad  h \, <\, h_c^{\cop} (\gl).
\end{equation}

\smallskip
{\bf Lower bound in Proposition \ref{prop:BG}.}
One may plug in \eqref{eq:borneinfh} the inequality $\tf(\gl,h)\geq -2\gl h$, which holds for $h\in\bbR$. This comes from the contribution of trajectories making one large excursion below the interface: $Z_{N,\gl,h}^{\go,\cop}\geq \bP(\tau_1\geq N) \bP(X_1=1) e^{-2\gl h N + 2\gl \sum_{n=1}^N \go_n}$, so that $\liminf_{N\to\infty} \frac1N \bbE\log Z_{N,\gl,h}^{\go,\cop} \geq -2\gl h$ (since the disorder is centered). One therefore has
\begin{equation}\label{eq:borneinfh2}
\sup_{a\geq 0}\Big\{ \tf^{\cop}(\gl, h - a ) - (1+\ga) \frac{a^2}{2\gU_{\infty} } \Big\} \geq \sup_{a\in\bbR} \Big\{  -2\gl (h-a)  - (1+\ga) \frac{a^2}{2\gU_{\infty}}\Big\} = -2\gl h +2\gl^2 \frac{\gU_{\infty}}{1+\ga}.
\end{equation}
Therefore, \eqref{eq:borneinfh} and \eqref{eq:borneinfh2} yield that $h_c^{\cop}(\gl)\geq  \frac{\gU_{\infty}}{1+\ga} \gl $ for all $\gl\geq 0$, which is the Monthus bound in Proposition \ref{prop:BG}, and gives the first part of the lower bound in Theorem \ref{thm:gapcop}.\qed

\smallskip
{\bf Lower bound in Theorem \ref{thm:gapcop}.}
More precise linear estimates on $\tf(\gl,h)$ give sharper lower bounds for $h_c^{\cop}(\gl)$.
The following lemma is analogous to \cite[Lem. 5.1]{cf:BCPSZ}. Important refinements were made to deal with correlations, with the help of the Gaussian nature of the disorder.
\begin{lemma}
\label{lem:linearbound}
If $\mu = \bE[\tau_1]<+\infty$, 
then for any $c\in\R$,
\begin{equation}\label{eq:gooa2}
\liminf_{\gl\downarrow 0} \frac{1}{\gl^2}\f^{\cop}(\lambda, c\lambda) \geq \frac12 \ccop  - c,
\end{equation}
where $\ccop$ has been introduced in \eqref{defccop}.
\end{lemma}

Let $(\gl_n)_{n\in\bbN}$ be a sequence of positive real numbers such that
\begin{equation}
\lim_{n\to\infty} \gl_n = 0\quad \mbox{ and }  \quad \lim_{n\to\infty} \frac{h_c^\cop(\gl_n)}{\gl_n} =  \liminf_{\gl \downarrow 0} \frac{h_c^\cop(\gl)}{\gl}.
\end{equation}
Simply using that $h\mapsto\f^{\cop}(\gl,h)$ is non-increasing and Lemma \ref{lem:linearbound}, one gets that
\begin{equation}
\liminf_{n\to \infty} \frac{1}{\gl_n^2} \f^{\cop}(\lambda_n, h_c^\cop(\gl_n)) \geq \frac12 \ccop  - \liminf_{\gl \downarrow 0} \frac{h_c^\cop(\gl)}{\gl}.
\end{equation}
Combining this with the smoothing inequality of Proposition \ref{pr:smoothcop}, 
one obtains, for all $u\geq 0$,
\begin{equation}
\label{eq:combine1}
\frac{1+\alpha}{2 \gU_{\infty}} u^2 \geq \liminf_{n\to\infty} \frac{1}{\gl_n^2} \f^{\cop}(\gl_n,h_c^{\cop}(\gl_n)-u\lambda_n) \geq  \frac12 \ccop - \liminf_{\lambda\downarrow 0} \frac{h_c^{\cop}(\gl)}{\gl} +u.
\end{equation} 
Then,
\begin{equation}
\label{eq:combine2}
\liminf_{\lambda\downarrow 0} \frac{h_c^{\cop}(\gl)}{\gl} \geq \frac12 \ccop + \sup_{u\geq 0} \left\{u-\frac{1+\alpha}{2\gU_{\infty}} u^2 \right\} =\frac12 \ccop + \frac{\gU_{\infty}}{2(1+\ga)}.
\end{equation}
 \qed
 
\begin{proof}[Proof of Lemma \ref{lem:linearbound}.]
Here, we adapt an idea of \cite[Theorem 6.3]{cf:G1}.
Let
\begin{equation}
\cN_N:=\max\{n\in \bbN_0\colon \tau_n\leq N\} = |\tau \cap \{1,\ldots, N\}|.
\end{equation}
Simply using Jensen's inequality on the (free) partition function, and then the law of large numbers for $\cN_N$, we obtain
\begin{multline}
\label{Jensen1}
	\f^{\cop}(\lambda,h) \geq
	\lim_{N\to \infty} \frac{1}{N}\,\bbE\log \E\Big[ \prod_{j=1}^{\cN_N}
	\frac{1+e^{-2\lambda\sum_{n=\tau_{j-1}+1}^{\tau_j} (\go_n+ h) } }{2}\Big]\\
	\geq \frac{1}{\mu} \bbE\E \left[\log\left( \frac{1+e^{-2\lambda\sum_{n=1}^{\tau_1} 
	(\go_n+ h) } }{2} \right) \right].
\end{multline}
If $h=c\gl$ and $k$ is fixed, a Taylor expansion gives
\begin{multline}
\bbE \bigg[ \log\left(\frac{1+e^{-2\lambda\sum_{n=1}^{k} 
	(\go_n+ h) } }{2} \right) + \lambda\sum_{n=1}^k(\go_n+ h) \bigg]\\
	\stackrel{\gl\downarrow 0}{\sim}\frac12\lambda^2 \bbE\bigg[ \left(\sum_{n=1}^k \omega_n\right)^2\bigg] = \frac12 \gl^2 \sum_{n,m=1}^k \rho_{nm}.
\end{multline}
 We can then apply Fatou's lemma (note that the expression in the expectation is non-negative) and get
 \begin{equation}
 \label{Fatou}
\liminf_{\gl\downarrow 0} \frac{1}{\gl^2}\frac{1}{\mu}\E\bbE \left[\log\left( \frac{1+e^{-2\lambda\sum_{n=1}^{\tau_1} 
	(\go_n+ h) } }{2} \right) + \lambda\sum_{n=1}^N(\go_n+h) \right]
	 \geq \frac{1}{2\mu}\bE\left[ \sum_{n,m=1}^{\tau_1}  \rho_{mn} \right].
 \end{equation}
Therefore, combining \eqref{Jensen1} with \eqref{Fatou} (recall that the $\go_n$'s are centered and that $h=c\gl$), one obtains
\begin{equation}
\liminf_{\gl\downarrow 0}\frac{1}{\gl^2}\f^{\cop}(\gl,c\gl)\geq  \frac{1}{2\mu}\bE\bigg[ \sum_{n,m=1}^{\tau_1} \rho_{nm} \bigg] - c = \frac12 \ccop  - c.
\end{equation}
\end{proof}

\section{On the copolymer model: upper bound}
\label{sec:uppercopol}
In this section we prove the upper bound part of Theorem \ref{thm:gapcop}, that is 
\begin{equation}
\limsup_{\gl\downarrow 0} \frac{h_c^{\cop}(\gl)}{\gl} \leq \max \bigg\{ \frac{\gU_{\infty}}{1+\ga}\, ;\, \frac12\frac{\gU_{\infty}}{1+\alpha} +\frac12 \ccop \bigg\}.
\label{eq:upboundcop}
\end{equation}
\smallskip

We use an idea from \cite{cf:DGLT}, later improved in \cite{cf:GLT1, cf:GLT2, cf:BCPSZ} and known as the fractional moment method.
To prove that $h_c^{\cop}(\gl)  \leq h_0$ (with $h_0=u\gl$), one needs to prove that $\tf^{\cop}(\gl,h_0)=0$. It is actually enough to show that for some $\zeta\in(0,1)$,
\begin{equation}
\label{eq:condfracmoment}
\liminf_{N\to\infty} \bbE[(Z^{\go,\cop}_{N,\gl,h_0})^{\zeta}] <+\infty.
\end{equation}
Indeed, by Jensen's inequality,
\begin{equation}
\label{eq:condfracmoment2}
 \frac{1}{N}\bbE \log Z^{\go,\cop}_{N,\gl,h_0} = \frac{1}{\zeta N}\bbE \log (Z^{\go,\cop}_{N,\gl,h_0})^{\zeta}
 \leq \frac{1}{\zeta N} \log \bbE[(Z^{\go,\cop}_{N,\gl,h_0})^{\zeta}],
\end{equation}
If \eqref{eq:condfracmoment} holds then we get $\tf(\gl,h_0)=0$ by letting $N$ go to $+\infty$ in \eqref{eq:condfracmoment2}, thus $h_c^{\cop}(\gl)  \leq h_0$.

\medskip
The proof consists in two main steps: (a) we estimate the fractional moment on the finite length scale $1/\gl^2$ (Section \ref{sec:chgmeas}) and (b) we glue finite-size estimates together thanks to a coarse-graining argument (Section \ref{sec:coarse}).
The qualitative picture is the following. First, Lemma \ref{lem:finitefraccop} suggests that blocks of length $t/\gl^2$ have a negative contribution as long as $u \geq \frac{\zeta}{2} \gU_{\infty} +\frac12 \ccop$. Another bound, namely $u\geq \zeta \gU_{\infty}$, is then necessary to control the contribution of large excursions below the interface and between two consecutive visited blocks. Besides, we need $\zeta>1/(1+\ga)$ in order to make the coarse-graining argument work. All in all, we get the condition $u\geq \max\{\frac12 \frac{\gU_{\infty}}{1+\ga} +\frac12 \ccop ; \frac{\Upsilon_\infty}{1+\ga}\}$.

\subsection{Finite-size estimates of fractional moments}
\label{sec:chgmeas}

To increase readibility, we shall often omit the symbol of the integer part.

\begin{lemma}
\label{lem:finitefraccop}
For any constants $u\in\bbR$ and $t>0$, one sets $h_{\gl}=  u\gl$ and $k_{\gl}=t/\gl^2$.  Then
\begin{equation}
\limsup_{\gl\downarrow0} \bbE\left[ \left( Z_{k_{\gl},\gl,h_{\gl}}^{\go, \cop} \right)^{\zeta} \right] \leq
	\exp \bigg( \zeta \left(  \frac12 \ccop + \frac{\zeta}{2} \gU_{\infty} - u  \right) t \bigg).
\end{equation}
\end{lemma}

\begin{proof}
Let $\tilde\bbP_{\delta,k}$ be the law obtained from $\bbP$ by tilting $(\go_1,\ldots,\go_k)$ by $\gd$, that is,
\begin{equation}
\label{eq:def.change.meas}
\frac{\dd\tilde\bbP_{\delta,k}}{\dd\bbP}(\go) = \frac{e^{\delta\sum_{i=1}^k \omega_i}}{\bbE\big[e^{\delta\sum_{i=1}^k \omega_i} \big]}.
\end{equation}
Using H\"older's inequality, one has for all $\zeta \in (0,1)$,
\begin{equation} \label{eq:trick}
	\bbE\big[ \big( Z_{k,\gl,h}^{\omega,\cop} \big)^\zeta \big]
	= \tilde\bbE_{\delta,k}\bigg[
	\big( Z_{k,\gl,h}^{\omega,\cop} \big)^\zeta \, \frac{\dd\bbP}{\dd\tilde\bbP_{\delta,k}} \bigg]
	\,\le\, \big[\tilde\bbE_{\gd,k}  (Z_{k,\beta,h}^{\omega,\cop}) \big]^\zeta
	\, \bigg[\tilde\bbE_{\delta,k}  \bigg( \frac{\dd\bbP}{\dd\tilde\bbP_{\delta,k}}
	\bigg)^{\frac{1}{1-\zeta}} \bigg]^{1-\zeta}.
\end{equation}
We also pick $\gd=a\gl$ for some constant $a$ that we optimize later on.

\medskip \noindent
{\bf (1)} Let us first estimate the second term in \eqref{eq:trick}, for large $k$ (that is small $\gl$). We have
\begin{equation}
\frac{\dd\bbP}{\dd\tilde\bbP_{\delta,k}}
= \exp\left( -\gd \langle \gU_k^{-1} \go, \ind_k\rangle  
    +\frac12 \gd^2 \langle \gU_k^{-1}\ind_k, \ind_k\rangle\right),
\end{equation}
where $\ind_k$ is a vector consisting of $k$ $1$'s. Therefore,
\begin{multline}
\label{eq:costshift}
	\tilde\bbE_{\delta,k} \bigg[ \bigg( \frac{\dd\bbP}{\dd\tilde\bbP_{\delta,k}}
	\bigg)^{\frac{1}{1-\zeta}} \bigg]^{1-\zeta} = \bbE\bigg[ \bigg( \frac{\dd\bbP}{\dd\tilde\bbP_{\delta,k}} \bigg)^{\frac{\zeta}{1-\zeta}} \bigg]^{1-\zeta} \\
	=\,
	\bbE\bigg[  e^{-\frac{\zeta \gd}{1-\zeta} \langle \gU_k^{-1}\go, \ind_k\rangle}  \bigg]^{1-\zeta}
	e^{ \frac{\zeta \gd^2}{2} \langle \gU_k^{-1} \ind_k,\ind_k \rangle } 
	= e^{\frac12 \frac{\zeta \gd^2}{1-\zeta} \langle \gU_k^{-1}\ind_k, \ind_k\rangle}.
\end{multline}
Using Assumption \ref{hyp:correlations}, $\langle \gU_\ell^{-1} \ind_\ell, \ind_\ell \rangle = (1+o(1)) \ell  \gU_{\infty} ^{-1} $ when $\ell$ goes to infinity, see \cite[Lemma~A.1]{cf:B13}. We obtain (recall that $k=k_{\gl}=t/\gl^2$ and $\gd=a\gl$):
\begin{equation}
\gd^2 \langle \gU_k^{-1} \ind_k, \ind_k\rangle \stackrel{\gl\downarrow 0}{=} t a^2 \gU_{\infty}^{-1} +o(1),
\end{equation}
which, in combination with \eqref{eq:costshift}, gives
\begin{equation}
\label{eq:costshift2}
\lim_{\gl\downarrow 0}\tilde\bbE_{\delta,k} \bigg[ \bigg( \frac{\dd\bbP}{\dd\tilde\bbP_{\delta,k}}
	\bigg)^{\frac{1}{1-\zeta}} \bigg]^{1-\zeta} = \exp\left(\frac12 \frac{\zeta}{1-\zeta} t a^2 \gU_{\infty}^{-1} \right).
\end{equation}

\medskip \noindent
{\bf (2)} For the first factor in \eqref{eq:trick},
since $h=h_{\gl} = u\gl$, one has
\begin{equation}
\tilde\bbE_{\delta, k}
	\big[ Z_{k,\gl,h}^{\omega,\cop} \big]  =\E\left[e^{-2\gl^2 (u -a)\sum_{i=1}^k  \gD_i  + 2\gl^2 \sum_{i,j=1}^k \rho_{ij} \gD_i\gD_j}\ind_{\{k\in\tau\}} \right].
\end{equation}
Recall that $\gl^2=t/k$. By Lemma \ref{lem:convergenceapp} (which gives the a.s.\ limits $\lim_{n\to\infty} \frac1n \sum_{i=1}^n \gD_i =\frac12$ and  $\lim_{n\to\infty} \frac1n \sum_{i,j=1}^n \rho_{ij} \gD_i \gD_j =\frac14 \ccop + \frac14 \gU_{\infty}$), one gets
\begin{equation}
\label{eq:modifmeascop}
	\lim_{\gl\downarrow 0}\tilde\bbE_{\delta, k}
	\big[ Z_{k,\, \gl,\, h}^{\omega,\cop} \big] 
	=\frac{1}{\mu} \exp\left( t \left\{a - u + \frac12\ccop +\frac12 \gU_{\infty}\right\} \right) \,.
\end{equation}
Combining \eqref{eq:trick} with \eqref{eq:costshift} and \eqref{eq:modifmeascop}, we get
\begin{equation} 
	\limsup_{\gl\downarrow 0}
	\bbE\big[ \big( Z_{k,\gl,h}^{\omega, \cop} \big)^\zeta \big]
	\,\le\, 
	\exp\left(\zeta t \left\{\frac12 \gU_{\infty} + \frac12\ccop - u + a +\frac{a^2}{2(1-\zeta)\gU_{\infty}} \right\} \right)
	\end{equation}
The optimal choice $a=\gU_{\infty}(\zeta - 1)$ finally gives Lemma \ref{lem:finitefraccop}.
\end{proof}

\subsection{The coarse-graining procedure}
\label{sec:coarse}
We proceed through several steps.

\medskip
{\bf STEP 0: Preliminaries.}
Let us abbreviate
\begin{equation}
\label{eq:def.c_gep}
c_0 = \max\left\{\frac{\gU_{\infty}}{1+\ga}; \frac{\gU_{\infty}}{2(1+\ga)} + \frac12 \ccop \right\},\qquad c_\gep = c_0 + \gep,\quad \gep>0.
\end{equation}
We need to show that there exists $\gl_0=\gl_0(\gep)$ sufficiently small
such that $\tf^{\cop}(\gl,c_{\gep}\gl)=0$ for $\gl\in(0,\gl_0)$. As explained at the beginning of this section, it is actually enough to find $\zeta_{\gep}\in(0,1)$ such that
\begin{equation}
\label{eq:condfracmoment2bis}
\liminf_{N\to\infty} \bbE[(Z^{\go,\cop}_{N,\gl,c_{\gep}\gl})^{\zeta_{\gep}}] <+\infty.
\end{equation}

We now fix $\gep>0$ and set
\begin{equation}
\label{eq:choice.par.cop}
h_{\gep}=h_{\gep}(\gl)=c_{\gep}\gl,\qquad \zeta_{\gep} = \frac{1}{1+\alpha} +\gep^2.
\end{equation}
The coarse-graining correlation length is chosen to be
$k=k_{\gl,\gep}=t_{\gep}/\gl^2$, where $t_{\gep}$ is a (large) constant, whose value is specified at the end of the proof.

In the rest of this section, the constants $C_0,C_1,...$ do not depend on $\gep$. We shall add a subscript $\gep$ for constants that do depend on $\gep$.

\bigskip
{\bf STEP 1: Setting up the coarse-graining.}
The size of the system is going to be a multiple of the correlation
length: $N = m\, k_{\gl,\gep}$, where $m\in\N$ is the macroscopic size.
The system is then partitioned into $m$ blocks $B_1, \ldots, B_m$
of size $k_{\gl,\gep}$, defined by
\begin{equation}
	B_i \,:=\, \big\{(i-1)\, k_{\gl,\gep} + 1\, , \ldots,\, i\, k_{\gl,\gep}\big\}
	\,\subseteq\, \{1,\ldots,N\} \,,
\end{equation}
so that the macroscopic (coarse-grained) ``configuration space'' is $\{1,\ldots,m\}$.
A macroscopic configuration is then a subset $J \subseteq \{1,\ldots,m\}$.
Let us define for $a,b\in\bbN_0$ with $a<b$:
\begin{equation}
\label{eq:z.onestretch}
z_a^b = \E\Big[ \exp\Big\{-2\gl \sum_{n=a+1}^b (\go_n+h)\gD_n \Big\}\,  \big| \, \exists k\geq 0:\, \tau_k=a, \tau_{k+1}=b \Big],
\end{equation}
which is the contribution of a large excursion between $a$ and $b$.
By decomposing the partition function according to the blocks visited by the polymer, we get
\begin{equation}
	Z_{N,\gl, c_{\gep}\gl }^{\omega,\cop} \,=\, \sum_{J \subseteq \{1,\ldots, m\}:
	\ m \in J} \hat Z_J,
\end{equation}
where for $J = \{j_1, \ldots, j_\ell\}$, with
$1 \le j_1 < j_{2} < \ldots < j_\ell = m$ and $\ell = |J|$, we set
\begin{equation}
\label{eq:1stdecomp}
	\hat Z_J  \,:=\, \sumtwo{d_1, f_1 \in B_{j_1}}{d_1 \le f_1}
	\ldots \sumtwo{d_{\ell-1}, f_{\ell-1} \in B_{j_{\ell-1}}}{d_{\ell - 1} \le f_{\ell - 1}}
	\sum_{d_\ell \in B_{j_\ell} = B_m}
	\Bigg( \prod_{i=1}^\ell K(d_i - f_{i-1})  z_{f_{i-1}}^{d_i} Z_{d_i, f_i} \Bigg) \,,
\end{equation}
with the conventions
\begin{equation}
\label{eq:Zdf}
f_0 = 0,\qquad  f_\ell = N, \qquad Z_{d_i, f_i} = Z_{f_i - d_i,\gl,k_\gep}^{\gt^{d_i}\go, \cop}.
\end{equation}
It follows that
\begin{equation} \label{eq:fracmom}
	\bbE \big[ \big( Z_{N,\gl, c_{\gep} \gl}^{\omega,\cop} \big)^{\zeta_\gep} \big]
	\,\le\, \sum_{J \subseteq \{1,\ldots, m\}:
	\ m \in J} \bbE \big[ \big( \hat Z_J \big)^{\zeta_\gep} \big] \,.
\end{equation}

\smallskip
We now focus on providing an upper bound on $\bbE \big[ \big( \hat Z_J \big)^{\zeta_\gep} \big]$.
Defining $\bar f_i = j_i k_{\gl,\gep}$ and $\bar d_i = (j_i -1) k_{\gl,\gep}$ for $i\in \{1,\ldots,\ell\}$, notice that
\begin{equation}
\label{eq:smallz.dec}
z_{f_{i-1}}^{d_i}\leq 4z_{f_{i-1}}^{\bar f_{i-1}} z_{\bar f_{i-1}}^{\bar d_{i}} z_{\bar d_{i}}^{d_{i}},
\end{equation}
 cf.\ \cite[Equation (3.16)]{cf:T1}. Let us then define
\begin{equation}
\label{eq:Zcheck}
\check Z_J \,:= \, \sumtwo{d_1, f_1 \in B_{j_1}}{d_1 \le f_1}
	\ldots \sumtwo{d_{\ell-1}, f_{\ell-1} \in B_{j_{\ell-1}}}{d_{\ell - 1} \le f_{\ell - 1}}
	\sum_{d_\ell \in B_{j_\ell} = B_m}
\prod_{i=1}^\ell K(d_i - f_{i-1})  z_{f_{i-1}}^{\bar f_{i-1}}   Z_{d_i, f_i} z_{\bar d_{i}}^{d_{i}} \, ,
\end{equation}
so that $\hat Z_J \leq 4^{\ell} \check Z_J \prod_{i=1}^\ell  z_{\bar f_{i-1}}^{\bar d_{i}}$
To decouple $\check Z_J$ and the $z_{\bar f_{i-1}}^{\bar d_{i}}$ 's we use the following lemma, which relies on Lemma \ref{lem:decouple.general} and is proved in Appendix \ref{sec:decouple.app}.
\begin{lemma}
\label{lem:decoupling0}
There exists $\gl_1=\gl_1(\gep)>0$ such that, for all $\gl\leq \gl_1$,
\begin{equation}
\bbE\Big[ \Big( \prod_{i=1}^\ell  z_{\bar f_{i-1}}^{\bar d_{i}}\Big)^{\zeta_\gep}  (\check Z_J)^{\zeta_\gep}\Big] \leq 2^\ell \bbE\Big[ \prod_{i=1}^\ell \Big( z_{\bar f_{i-1}}^{\bar d_{i}}\Big)^{\zeta_\gep} \Big] \bbE \Big[ (\check Z_J)^{\zeta_\gep} \Big].
\end{equation}
\end{lemma}
We therefore get, using that $4^{\zeta_{\gep}}\leq 4$
\begin{equation}
\label{eq:Zhat2}
\bbE \Big[ ( \hat Z_J )^{\zeta_\gep} \Big]
\leq 8^{\ell} \bbE\Big[ \prod_{i=1}^{\ell} \Big( z_{\bar f_{i-1}}^{\bar d_{i}} \Big)^{\gz_{\gep}}\Big]  \bbE\Big[ (\check Z_J )^{\gz_{\gep}} \Big].
\end{equation}
We first estimate the term $\bbE\big[\prod_{i=1}^{\ell}  (z_{\bar f_{i-1}}^{\bar d_{i}} )^{\zeta_{\gep}} \big]$. Let us write $z_{\bar f_{i-1}}^{\bar d_{i}} = \frac12(1 +  e^{\cH_{i}})$, with $\cH_{i}:= -2\lambda \sum_{u=\bar f_{i-1}+1}^{\bar d_{i}} (\omega_u + c_{\gep} \lambda)$. Then, since $\zeta_{\gep}\leq 1$, we get that $(z_{\bar f_{i-1}}^{\bar d_{i}} )^{\zeta_{\gep}}\leq 2^{-\zeta_{\gep}} (1+e^{\zeta_{\gep} \cH_{i}})$. Using a binomial expansion of $\prod_{i=1}^{\ell}  (1+e^{\zeta_{\gep} \cH_{i}})$, we obtain
\begin{equation}
\label{eq:binomexpansionsmallz}
\bbE\Big[ \prod_{i=1}^{\ell} \big( z_{\bar f_{i-1}}^{\bar d_{i}} \big)^{\gz_{\gep}}\Big] \leq (2^{-\zeta_{\gep}})^{\ell} \sum_{I\subseteq \{1,\ldots \ell\}} \bbE\Big[\prod_{i\in I} \exp(\zeta_\gep \cH_i )\Big],
\end{equation}
where the product in the right-hand side is $1$ when $I=\emptyset$. Recalling the definition of $\cH_{i}$, we may write
\begin{equation}
\label{eq:productsmallz}
\bbE\Big[\prod_{i\in I} \exp( \zeta_{\gep} \cH_{i} )\Big] \leq  \exp\bigg\{ 2\gl^2\zeta_{\gep}  \sum_{i\in I}  \bigg(\zeta_{\gep}\!\!\!\! \sum_{u,v = \bar f_{i-1} +1}^{\bar d_i}\!\!\!\!  \rho_{uv} +\zeta_{\gep}\!\!\!\! \sumtwo{u \in (\bar f_{i-1}+1,\bar d_i]}{v \notin (\bar f_{i-1}+1, \bar d_i]} \!\! |\rho_{uv}|   -  c_\gep (\bar d_i -\bar f_{i-1}) \bigg) \bigg\}.
\end{equation}
Since the correlations are summable, there exists $n_1=n_1(\gep)$ such that, for $n\geq n_1$, one has $\sum_{u,v = 1}^{n}\rho_{uv} \leq n (\gU_{\infty}~+~\gep/4)$ as well as $\sum_{u=1}^{n} \sum_{v\notin\{1,\ldots,n\}} |\rho_{uv}| \leq \gep n/4 $. Therefore, there exists $\gl_2=\gl_2(\gep)$ such that, for $\gl\leq \gl_2 $, one has $k_{\gl,\gep}\geq n_1$. Then, in \eqref{eq:productsmallz}, since either $\bar d_{i+1} - \bar f_i=0$ or $\bar d_{i+1} - \bar f_i\geq k_{\gl,\gep} \geq n_1$, we obtain
\begin{multline}
\label{eq:useMonthusbound}
\bbE\Big[\prod_{i\in I} \exp( \zeta_{\gep} \cH_i )\Big] \leq \prod_{i\in I} \exp\Big\{ 2\gl^2\zeta_{\gep}  ( \zeta_{\gep} \gU_{\infty}  + \gep/2 -  c_\gep )(\bar d_{i+1} - \bar f_i) \Big\} \\
= \prod_{i\in I} \exp\Big\{ 2\gl^2\zeta_{\gep} \Big( \frac{\gU_{\infty}}{1+\alpha} + \gep^2 \gU_{\infty} +\gep/2 - c_0 - \gep \Big)(\bar d_{i+1} - \bar f_i)\Big\} \leq 1,
\end{multline}
where, at first, we used the definitions of $\zeta_{\gep}$ and $c_{\gep}$, and then we chose $\gep$ small enough so that $\gep^2 \gU_{\infty} -\gep/2\leq 0$. Note that the condition $c_{0}\geq \gU_{\infty}/(1+\ga)$ (the so-called Monthus bound) is crucial here: otherwise the terms in \eqref{eq:useMonthusbound} would diverge when $d_i-f_{i-1}$ goes to infinity.
Plugging this estimate in \eqref{eq:binomexpansionsmallz} gives
\begin{equation}
\bbE\Big[ \prod_{i=1}^{\ell} \big( z_{\bar f_{i-1}}^{\bar d_{i}} \big)^{\gz_{\gep}}\Big] \leq 2^{-\zeta_{\gep}\ell} \sum_{I\subseteq \{1,\ldots,\ell\}} 1 = 2^{-\zeta_{\gep}\ell} \, 2^{\ell} \leq 2^{\ell} .
\end{equation}

Combining this with \eqref{eq:Zhat2}, we have for $\gl\leq \min\{\gl_1,\gl_2\}$,
\begin{equation}
\label{eq:Zhat3}
\bbE \big[ \big( \hat Z_J \big)^{\zeta_\gep} \big]
\leq (16)^{\ell} \bbE\big[ (\check Z_J )^{\gz_{\gep}} \big],
\end{equation}
and we are now left with estimating $\bbE\big[ (\check Z_J )^{\gz_{\gep}} \big]$.

\bigskip
{\bf STEP 2: Applying the change of measure.}
Recall the definition of the tilted measures given in \eqref{eq:def.change.meas}. Here, we denote by $\tilde \bbP_J$ the law obtained from $\bbP$ by tilting $\go_n$, for each $n\in \bigcup_{i\in J} B_i$, by
\begin{equation}
\label{eq:def.a_gep}
\gd=a_{\gep }\gl, \qquad {\rm where} \quad a_{\gep}:=-(1-\zeta_{\gep})\gU_{\infty}.
\end{equation}
Note that this value of $a_\gep$ is chosen according to the proof of Lemma \ref{lem:finitefraccop}. By H\"older's inequality,
\begin{equation}
\label{eq:Holder}
\bbE \big[ \big( \check Z_J \big)^{\zeta_\epsilon} \big] \leq \big[\tilde\bbE_J( \check Z_J)\big]^{\zeta_{\gep}} \tilde\bbE_J\bigg[\bigg(\frac{\dd \bbP}{\dd \tilde \bbP_J} \bigg)^{\frac{1}{1-\zeta_{\gep}}}\bigg]^{1-\zeta_{\gep}}.
\end{equation}

\smallskip \noindent
{\bf (1)} The second factor in \eqref{eq:Holder} is computable, since
\begin{equation}
\frac{\dd \bbP}{\dd \tilde \bbP_J} = \exp\Big( - \langle \gU_N^{-1} \go ,\gd \ind_J \rangle + \frac12 \langle \gU_N^{-1}\gd \ind_J,\gd \ind_J\rangle\Big),
\end{equation}
where $\ind_J$ is the indicator function of $\cup_{j\in J} B_j$. Similarly to \eqref{eq:costshift}, we have
\begin{equation}
\label{eq:computecghtmeas}
\tilde\bbE_J\bigg[\bigg(\frac{\dd \bbP}{\dd \tilde \bbP_J} \bigg)^{\frac{1}{1-\zeta_{\gep}}}\bigg]^{1-\zeta_{\gep}}
= \exp \bigg( \frac12 \frac{\zeta_\gep}{1-\zeta_\gep} \gd^2\langle \gU_N^{-1} \ind_J, \ind_J\rangle \bigg).
\end{equation}
We now estimate $\langle \gU_N^{-1} \ind_J,\ind_J\rangle$ using the idea of \cite[Lemma A.3]{cf:B13}. Define
\begin{equation}
 R_J = \gU_{\infty} \ind_J - \gU_N \ind_J .
\end{equation}
Because of the summability of correlations, the $\ell^1$-norm of $R_J$ is bounded from above by $\sum_{i\in J} \sum_{n\in B_i} \sum_{m\notin B_i} |\rho_{nm}|$, which is $|J| o(k_{\gep,\gl}) $ as $k_{\gep,\gl}$ goes to infinity.
Then, 
\begin{equation}
\gU_{\infty} \gU_N^{-1} \ind_J = \ind_J - \gU_N^{-1} R_J
\end{equation}
and
\begin{equation}
\label{eq:approxgUJJ}
\gU_{\infty} \langle \gU_N^{-1} \ind_J,\ind_J\rangle = |J|k_{\gl,\gep} - \langle \gU_N^{-1} R_J,\ind_J\rangle =(1+o(1)) |J| k_{\gl,\gep} .
\end{equation}
To obtain the last equality, we use that $\ind_{J} = \gU_{\infty}^{-1} (\gU_N \ind_J - R_J)$, and get
\begin{equation}
|\langle \gU_N^{-1} R_J,\ind_J\rangle|= \gU_\infty^{-1} (\langle R_J, \ind_J \rangle - \langle \gU_N^{-1} R_J, R_J \rangle)= |J|\, o(k_{\gl,\gep}).
\end{equation}
Indeed, $|\langle R_j,\ind_J \rangle| = |J|\, o(k_{\gep,\gl})$ and $|\langle \gU_N^{-1} R_J,R_J \rangle|\leq ||| \gU_N^{-1}|||\, || R_J ||_{\ell^1} = |J| \, o(k_{\gep,\gl}) $, thanks to  Assumption \ref{hyp:correlations}, which gives a uniform bound on $||| \gU_N^{-1}|||$.

Therefore, there exists $\gl_3=\gl_3(\gep)$ such that, for $\gl\leq \gl_3$, we have $\langle \gU_N^{-1} \ind_J,\ind_J\rangle \leq (1+\gep^2) \gU_{\infty}^{-1} |J| k_{\gl,\gep}$.  Plugging this in \eqref{eq:computecghtmeas}, and recalling that $\gd=-(1-\zeta_{\gep}) \gU_{\infty} \gl$, $k_{\gl,\gep}=t_{\gep}/\gl^2$,  we have, for $\gl\leq \gl_3$,
\begin{equation}
\label{eq:part1-chgtmeas}
\tilde\bbE_J\bigg[\bigg(\frac{\dd \bbP}{\dd \tilde \bbP_J} \bigg)^{\frac{1}{1-\zeta_{\gep}}}\bigg]^{1-\zeta_{\gep}}\leq \exp\left(  (1+\gep^2) \frac{\zeta_{\gep}}{2}  (1-\gz_{\gep})\gU_{\infty} t_{\gep} \ell \right) \leq (C_{\gep})^{\ell},
\end{equation}
with 
\begin{equation}
C_{\gep} = \exp \bigg( t_{\gep} \gz_{\gep} (1+\gep^2) \frac{\ga \gU_{\infty}}{2(1+\ga)} \bigg), 
\end{equation}
(we used that $1-\gz_{\gep} \leq 1-1/(1+\ga) = \ga/(1+\ga)$).

\smallskip \noindent
{\bf (2)} Let us now deal with the first factor in \eqref{eq:Holder}, that is $\tilde\bbE_J [ \check Z_J]$. From \eqref{eq:Zcheck}, we need to estimate in particular
\begin{equation}
\label{eq:step2.2}
\tilde \bbE_J \left[  \prod_{i=1}^\ell z_{f_{i-1}}^{\bar f_{i-1}}   Z_{d_i, f_i} z_{\bar d_{i}}^{d_{i}}  \right]
\end{equation}
for every $d_i\leq f_i$ in $B_{j_i}$, where $j_i\in J$. An extra difficulty comes from the lack of independence of the $\omega$'s. The following lemma, proved in Appendix \ref{sec:decouple.app} thanks to Lemma \ref{lem:decouple.general}, allows to decouple the factors in the product above.
\begin{lemma}\label{lem:decouple}
There exists $\gl_1=\gl_1(\gep)$ (the same as in Lemma \ref{lem:decoupling0}) such that, for all $\gl\leq \gl_1$,
\begin{equation}
\tilde\bbE_J \bigg[  \prod_{i=1}^\ell z_{f_{i-1}}^{\bar f_{i-1}}   Z_{d_i, f_i} z_{\bar d_{i}}^{d_{i}}  \bigg] \leq 8^\ell \prod_{i=1}^\ell \tilde \bbE_J \big[ z_{f_{i-1}}^{\bar f_{i-1}} \big]
\tilde \bbE_J \big[ Z_{d_i, f_i}\big] \tilde \bbE_J \big[ z_{\bar d_{i}}^{d_{i}} \big].
\end{equation}
\end{lemma}
Let us now estimate the different terms in this product.

\smallskip \noindent
{\bf (a)} We first take care of the first and third terms, that is $\tilde \bbE_J \big[ z_{0}^n \big]$, setting $n=\bar f_{i-1}- f_{i-1}$ or $n= d_i - \bar d_i$. Writing $z_0^n=\frac12+\frac12 \exp\big( -2\gl\sum_{i=1}^n (\go_i - c_{\gep}\gl) \big)$, one directly has
\begin{equation}
\tilde \bbE_J \big[ z_{0}^n \big] = \frac12 +\frac12 \exp\bigg( 2\gl^2n  (a_\gep -c_{\gep}) +2\gl^2 \sum_{i,j=1}^n \rho_{ij} \bigg) .
\end{equation}
Note that $a_{\gep}-c_{\gep} \leq -\gU_{\infty} + \gep^2 \gU_{\infty} -\gep \leq -\gU_{\infty} - 3\gep/4$, provided that $\gep$ is chosen small enough (we used here that $c_{\gep}\geq \frac{\gU_{\infty}}{1+\ga}$). There exists $n_1=n_1({\gep})$ such that for $n\geq n_{1}$, $\sum_{i,j =1}^n \rho_{ij} \leq n( \gU_{\infty}+\gep/2)$. For $n\geq n_2$, one has
\begin{equation}
\tilde \bbE_J \big[ z_{0}^n \big]\leq \frac12 + \frac12 \exp\bigg( 2\gl^2 n (-\gU_{\infty} - 3\gep/4 + \gU_{\infty} +\gep/2)\bigg)  \leq 1.
\end{equation}
On the other hand, for $n\leq n_{1}$, the uniform bound $\sum_{i,j =1}^n \rho_{ij} \leq n \sum_{k\in\bbZ} | \rho_k |$ holds, and 
\begin{equation}
\tilde \bbE_J \big[ z_{0}^n \big] \leq \frac12 +\frac12 \exp\big( 2\gl^2 n_{1}  \sum_{k\in\bbZ} | \rho_k |  \big) .
\end{equation}
Therefore, for all $\gl \leq\gl_4(\gep) = 1/\sqrt{n_1(\gep)}$, we have
\begin{equation}
\label{eq:bounded.z0n}
 \tilde \bbE_J \big[ z_{0}^n \big]\leq C_1 ,\qquad \rm{where}\quad C_1:= e^{2 \sum_{k\in\bbZ} | \rho_k |}.
\end{equation} 

Using Lemma \ref{lem:decouple}, one therefore has for $\gl\leq \min\{\gl_1,\gl_4\}$,
\begin{equation}
\label{eq:eachtermZcheck}
\tilde\bbE_J \bigg[  \prod_{i=1}^\ell z_{f_{i-1}}^{\bar f_{i-1}}   Z_{d_i, f_i} z_{\bar d_{i}}^{d_{i}}  \bigg] \leq (8 C_1^2)^\ell \prod_{i=1}^\ell 
\tilde \bbE_J \big[ Z_{d_i, f_i}\big].
\end{equation}

\smallskip \noindent
{\bf (b)} We now turn to $\tilde\bbE_J\big[Z_{0, n}\big]$ for $n=f_i-d_i\in \{0,\ldots, k_{\gl,\gep}\}$. Bounding $\ind_{\{n\in\tau\}}$ by $1$, we have
\begin{equation}
\label{eq:shift.expectation}
\tilde\bbE_J [ Z_{0,n}] \leq \E\Big[\exp \Big\{-2\gl^2 (c_{\gep} -a_{\gep})\sum_{i=1}^n  \gD_i  + 2\gl^2 \sum_{i,j=1}^n \rho_{ij} \gD_i\gD_j \Big\}\Big].
\end{equation}
We now distinguish according to the value of $n$.

(i) Case $n\in \{(1-\gep^2)k_{\gl,\gep},\ldots, k_{\gl,\gep}\}$. By Lemma \ref{lem:convergenceapp}, there exists $\gl_5=\gl_5(\gep)>0$ such that, for $\gl\leq \gl_5$ and $n\in \{(1- \gep^2) k_{\gl,\gep},\ldots, k_{\gl,\gep}\}$, we have (recall $k_{\gl,\gep} =t_{\gep}/\gl^2$)
\begin{equation}
\tilde\bbE_J [ Z_{0,n} ] \leq \exp \Big\{ (1-\gep^2)t_{\gep} \Big(a_\gep - c_\gep +
\frac{\ccop}{2} + \frac{\gU_\infty}{2}\Big) \Big\}.
\end{equation}
The fact that the bound is uniform comes from the Lipschitz character of the functions $x\mapsto e^{C x}$, when $C$ and $x$ both range over a bounded set. Recall the definitions of $c_{\gep}$ and $a_{\gep}$ in \eqref{eq:def.c_gep} and \eqref{eq:def.a_gep}, a straightforward computation gives
\begin{equation}
\label{eq:playwithaepsceps}
 a_\gep - c_\gep + \frac{\ccop}{2} + \frac{\gU_\infty}{2} \leq  -\frac{\ga \gU_{\infty}}{2(1+\ga)}- \gep +\gep^2 \gU_{\infty}\leq -\frac{\ga \gU_{\infty}}{2(1+\ga)} - \gep/2,
\end{equation}
where in the last inequality we took $\gep$ small enough.
In the end, we get that for $\gl\leq \gl_5$
\begin{equation}
\label{eq:tildeZgood}
\forall n\in \{(1- \gep^2) k_{\gl,\gep},\ldots, k_{\gl,\gep}\}, \quad \tilde\bbE_J\big[ Z_{0,n}\big] \leq \exp\Big\{ - t_{\gep} \, \Big(\frac{\ga \gU_{\infty}}{2(1+\ga)}+ \gep/2\Big) \Big\}.
\end{equation}

(ii) General bound for $n\leq k_{\gl,\gep}$: we show that $\tilde\bbE_J[Z_{0,n}]\leq C_1$ for all $n\in\{1,\ldots,k_{\gl,\gep}\}$, provided that $\gl$ is small enough (with the constant $C_1$ defined in \eqref{eq:bounded.z0n}).
Indeed, by Lemma \ref{lem:convergenceapp}, there exists $n_2=n_2(\gep)$ such that, for $n\geq n_2$,
\begin{gather}
\bP \bigg( \frac{1}{n}\sum_{i=1}^n  \gD_i  \leq 1/2-\gep^2 \bigg) \leq \frac12 e^{-2t_{\gep} \sum_{k\in\bbZ} |\rho_k|} \label{eq:STEP2bii.one} \\
\bP \bigg( \frac1n \sum_{i,j=1}^n \rho_{ij} \gD_i\gD_j \geq \frac14 \ccop +\frac14 \gU_{\infty} +\gep^2 \bigg)\leq \frac12 e^{-2t_{\gep} \sum_{k\in\bbZ} |\rho_k|}. \label{eq:STEP2bii.two}
\end{gather}
Therefore, if $n_2 \leq n\leq k_{\gl,\gep}$, let us decompose the expectation in \eqref{eq:shift.expectation} according to whether $\frac{1}{n}\sum_{i=1}^n  \gD_i \geq 1/2-\gep^2$ and $\frac1n \sum_{i,j=1}^n \rho_{ij} \gD_i\gD_j \leq \frac14 \ccop +\frac14 \gU_{\infty} +\gep^2 $ or not, which gives
\begin{multline}
\tilde\bbE_J[ Z_{0,n}] \leq e^{-\gl^2n \{(c_{\gep}-a_{\gep})(1-2\gep^2) + \frac12 \ccop +\frac12 \gU_{\infty} +2\gep^2\} } +  e^{-2t_{\gep} \sum_{n\in\bbZ} |\rho_n|}  e^{2t_{\gep} \sum_{n\in\bbZ} |\rho_n|}.
\end{multline}
For the second part of the sum, we used the uniform upper bounds $2\gl^2 (c_{\gep} -a_{\gep})\sum_{i=1}^n  \gD_i \geq 0$, $\sum_{i,j=1}^n \rho_{ij} \gD_i\gD_j\leq \sum_{n\in\bbZ} |\rho_n|$, and $\gl^2 n\leq t_{\gep}$. Recalling \eqref{eq:playwithaepsceps}, and taking $\gep$ small enough, we conclude that $\tilde\bbE_J[ Z_{0,n}]\leq 2$ if $ n_2 \leq n \leq k_{\gl,\gep}$.

For the case $n\leq n_2(\gep)$, let us set $\gl_6=1/\sqrt{ n_2}$. Then, \eqref{eq:shift.expectation} gives the following uniform bound, for $\gl\leq \gl_6$ and $n\leq  n_2$,
\begin{equation}
\label{eq:tildeZgeneral}
\tilde\bbE_J [ Z_{0,n}] \leq \exp\Big( 2\gl^2 n_3 \sum_{n\in\bbZ} |\rho_n|\Big) \leq C_1=\exp\Big(2 \sum_{k\in\bbZ} | \rho_k | \Big).
\end{equation}

\medskip
To summarize, we collect the estimates in \eqref{eq:Zcheck}, \eqref{eq:eachtermZcheck}, \eqref{eq:tildeZgood} and \eqref{eq:tildeZgeneral}. We get that if $\gep$ is fixed and small enough, then there exists $\gl_7(\gep)=\min(\gl_i, i\in\{1,\ldots, 6\})$ such that for $\gl\leq \gl_7(\gep)$,
\begin{equation}
\label{eq:Zcheck2}
\tilde\bbE_J \big[ \check Z_J \big] \leq (C_2)^{ \ell}  \sumtwo{d_1, f_1 \in B_{j_1}}{d_1 \le f_1}
	\ldots \!\!\!\! \sumtwo{d_{\ell-1}, f_{\ell-1} \in B_{j_{\ell-1}}}{d_{\ell - 1} \le f_{\ell - 1}}
	\sum_{d_\ell \in B_{j_\ell} = B_m}
\prod_{i=1}^\ell 
K(d_i-f_{i-1}) U(d_i-f_i),
\end{equation}
where $C_2:= (8 C_1^3)^{\zeta_{\gep}}$, and
\begin{equation}
\label{defU}
U(n) =
\begin{cases}
D_{\gep} & \quad \text{ if } n \in \{(1-\gep^2)k_{\gl,\gep},\ldots, k_{\gl,\gep}\} ;\\
  1  & \quad \text{ for all } n < (1-\gep^2) k_{\gl,\gep} ;
\end{cases}
\end{equation}
with (recall \eqref{eq:tildeZgood})
\begin{equation}
D_{\gep}:=\exp\Big\{ - t_{\gep} \, \Big(\frac{\ga \gU_{\infty}}{2(1+\ga)}  + \gep/2\Big) \Big\}.
\end{equation}

\bigskip
{\bf STEP 3: Simplifying $\tilde \bbE_J [\check Z_J]$ as a coarse-grained system.} In this part we transform \eqref{eq:Zcheck2} into a partition function for a coarse-grained system, where the set $J$ plays the role of a renewal set.

Let us denote by $\oJ$ the set of indices $j\in J$ such that $j-1\in J\, \cup\{0\}$ and $j+1\in J\cup\{m+1\}$. Recall that we denote by $j_1,j_2,\ldots,j_\ell$ the elements of $J$.
If $j_i\in J\setminus \oJ$, we bound $U(d_i-f_i)$ uniformly by $1$.
On the other hand, if $j_i\in \oJ$, then $j_{i-1}\in J\, \cup \{0\}$, and as in \cite[Section 2, STEP 3]{cf:BCPSZ}, one shows that:

\smallskip
{\bf (Step 3.1)} The main contribution in \eqref{eq:Zcheck2} comes from the $d_i$'s and $f_{i-1}$'s such that $|d_i-f_{i-1}|\leq \gep^2 k_{\gl,\gep}/2$.
We recall this idea for the sake of completeness. Using that for all $n$, $U(n) \geq D_\gep$, we get
\begin{equation}
\label{eq:bound1}
\sumtwo{f_{i-1}\in B_{j_{i-1}}\, ,\,d_i\in B_{j_i}}{|d_i-f_{i-1}|\leq \gep^2 k_{\gl,\gep}/2}  U(f_{i-1}-d_{i-1}) K(d_i-f_{i-1}) U(f_{i}-d_{i}) \geq K(1) D_{\gep}^2
\end{equation}
where in the sum we kept only the term corresponding to $d_i-f_{i-1}=1$. Let us now turn to the sum restricted to $|d_i-f_{i-1}|> \gep^2 k_{\gl,\gep}/2$.
By \eqref{defK}, there is  a constant $C_3$ such that
\begin{equation}
\forall n\geq \gep^2 k_{\gl,\gep}, \qquad K(n)\leq C_3\, \varphi(\gep^2 k_{\gl,\gep}) (\gep^2 k_{\gl,\gep})^{-(1+\ga)}.
\end{equation}
Therefore,
\begin{multline}
\label{eq:bound2}
\sumtwo{f_{i-1}\in B_{j_{i-1}}\, ,\,d_i\in B_{j_i}}{|d_i-f_{i-1}|> \gep^2 k_{\gl,\gep}/2}  K(d_i-f_{i-1})\leq C_3\, \varphi(\gep^2 k_{\gl,\gep}) (\gep^2 k_{\gl,\gep})^{-(1+\ga)}  k_{\gl,\gep}^2\\ \leq  C_3\, \gep^{-2(1+\ga)} \varphi(\gep^2 k_{\gl,\gep}) k_{\gl,\gep}^{1-\ga},
\end{multline}
where we used that there were at most $k_{\gl,\gep}^2$ terms in the sum. Notice that $\varphi(\gep^2 k_{\gl,\gep}) k_{\gl,\gep}^{1-\ga}$ can be made arbitrarily small by choosing $k_{\gl,\gep}$ large (in the case $\ga =1$, one uses that $\lim_{n\to\infty}\varphi(n)=0$, see \cite[Proposition 1.5.9b]{cf:Bingham}).
Therefore there exists $\gl_8=\gl_8(\gep)$ such that for $\gl\leq \gl_8$,
\begin{equation}
C_3\, \gep^{-(1+\ga)} \varphi(\gep^2 k_{\gl,\gep}) k_{\gl,\gep}^{1-\ga} \leq K(1) D_{\gep}^2.
\end{equation}
Then, the sum in \eqref{eq:bound2} is smaller than the sum in \eqref{eq:bound1}, which proves our claim.

\smallskip
{\bf (Step 3.2)} If $\gl\leq \gl_8$, one may therefore restrict the summation in \eqref{eq:Zcheck2} to $|d_i-f_{i-1}|\leq \gep^2 k_{\gl,\gep}/2$ and $|d_{i+1}-f_{i}|\leq \gep^2 k_{\gl,\gep}/2$ whenever $j_i \in \oJ$, at the cost of an extra factor $2$ each time. The total factor thereby introduced in \eqref{eq:Zcheck2} is then smaller than $2^{\ell}$. Thus, for $j_i\in\oJ$, the summation is only over $d_i$'s and $f_i$'s verifying $f_i-d_i \geq (1-\gep^2) k_{\gl,\gep}$, and \eqref{defU} allows us to replace $U(f_i-d_i)$ by $D_{\gep}$. This is the crucial replacement. Dropping the restriction in the summation yields
\begin{equation}
\label{eq:Zcheck3}
\tilde\bbE_J \big[  \check Z_J  \big] \leq (2C_2)^{\ell} (D_{\gep})^{|\oJ|}\sumtwo{d_1, f_1 \in B_{j_1}}{d_1 \le f_1}
	\ldots \!\!\!\! \sumtwo{d_{\ell-1}, f_{\ell-1} \in B_{j_{\ell-1}}}{d_{\ell - 1} \le f_{\ell - 1}}
	\sum_{d_\ell \in B_{j_\ell} = B_m}
\prod_{i=1}^\ell 
K(d_i-f_{i-1}).
\end{equation}

Note that the sum above is now close to the probability of a renewal event. One could actually insert in \eqref{eq:Zcheck3} the terms $\bP(f_i-d_i \in \tau)$, each time only at the cost of a factor $C_4>0$. Indeed, the Renewal Theorem tells that $\lim_{n\to\infty} \bP(n\in\tau) =1/\mu>0$, so there exists a constant $C_4$ such that $\bP(n\in\tau) \geq 1/C_4$ for all $n$. We actually deal with this sum more directly. Since the $K(d_i - f_{i-1})$'s are the only terms containing $d_i$ and $f_{i-1}$, one can drop the restriction $f_{i}\geq d_{i}$ and get
\begin{equation}
\label{eq:Zcheck4}
\tilde\bbE_J \big[ \check Z_J \big] \leq (2C_2 )^{\ell} (D_{\gep})^{|\oJ|}\
\prod_{i=1}^\ell 
\sumtwo{f_{i-1\in} B_{j_{i-1}}}{ d_i \in B_{j_i} } K(d_i-f_{i-1}),
\end{equation}
where we used the notation $B_{j_0}=\{0\}$, so that $f_0=0$.

\smallskip
{\bf (Step 3.3)} Now, we estimate each sum, depending on whether $j_i - j_{i-1} =1$ or not.

\textbullet\ If $j_i - j_{i-1} =1$, then we can assume by translation invariance that $j_{i-1} =1$ and $j_i=2$.
One uses that $\bP(n\in\tau) \geq 1/C_4$ for all $n\in\bbN$, to get that

\begin{multline}
\label{eq:gapsmall}
\sum_{ f\in B_{1}, d\in B_2 } K(d-f) \leq C_4 \sum_{ f\in B_{1}, d\in B_2 } \bP(f\in\tau) K(d-f)  \\
\leq C_4 \bP(\tau\cap B_1 \neq \emptyset,\, \tau\cap B_2 \neq \emptyset) \leq C_4.
\end{multline}

\textbullet\  If $j_i-j_{i-1}\geq 2$, then $d_{i}-f_{i-1}\geq (j_{i}-j_{i-1}-1) k_{\gl,\gep}$. By \eqref{defK} there is a constant $C_5$ such that
\begin{equation}
K(d_i-f_{i-1}) \leq \frac{C_5}{(k_{\gl,\gep})^{1+\ga}} \frac{\varphi((j_{i}-j_{i-1}-1)k_{\gl,\gep}) }{(j_{i}-j_{i-1}-1)^{1+\ga}} \leq \frac{2^{1+\ga}C_5}{(k_{\gl,\gep})^{1+\ga}} \frac{\varphi((j_{i}-j_{i-1}-1)k_{\gl,\gep}) }{(j_{i}-j_{i-1})^{1+\ga}},
\end{equation}
where the last inequality holds because $n-1\geq n/2$ for $n\geq 2$. Then, since there are at most $k_{\gl,\gep}^2$ terms in the sum,
\begin{equation}
\sum_{ f_{i-1}\in B_{j_{i-1}}, d_i\in B_{j_i} } K(d_i-f_{i-1}) \leq \frac{2^{1+\ga}}{(j_{i}-j_{i-1})^{1+\ga}} (k_{\gl,\gep})^{1-\ga} \varphi((j_{i}-j_{i-1}-1)k_{\gl,\gep}).
\end{equation}
We need to deal with cases $\ga=1$ and $\ga>1$ in slightly diffferent ways.

{\bf (a)} If $\ga=1$, since $\mu=\bE[\tau_1]<+\infty$, necessarily $\lim_{n\to\infty} \varphi(n) = 0$, see \cite[Proposition 1.5.9b]{cf:Bingham}.
Moreover, $(j_i - j_{i-1}-1) k_{\gl,\gep} \geq k_{\gl,\gep} = t_\gep \gl^{-2}$. Therefore, there exists $\gl_9 = \gl_9(\gep)$ such that for $\gl\in(0,\gl_9)$,
\begin{equation}
\varphi((j_i - j_{i-1}-1) k_{\gl,\gep}) \leq 2^{-(1+\ga)} D_\gep^2,
\end{equation}
and one finds
\begin{equation}
\label{eq:sumalpha=1}
\sum_{ f_{i-1}\in B_{j_{i-1}}, d_i\in B_{j_i} } K(d_i-f_{i-1}) \leq D_\gep^2 (j_i - j_{i-1})^{-2}\leq\frac{ (D_{\gep})^2}{(j_{i}-j_{i-1})^{1+\ga - \gep^2/2}}.
\end{equation}

{\bf (b)} If $\ga>1$, we use that $\varphi(n)\leq C_6 n^{\gep^2/2} $, for some constant $C_6=C_6(\gep)$, and get
\begin{equation}
\label{eq:sumalpha>1}
\sum_{ f_{i-1}\in B_{j_{i-1}}, d_i\in B_{j_i} } K(d_i-f_{i-1}) \leq 2^{1+\ga}C_5 C_6 (k_{\gl,\gep })^{1-\ga +\gep^2/2}\frac{1}{(j_{i}-j_{i-1})^{1+\ga - \gep^2/2}}.
\end{equation}
Then, one can choose $\gep$ small enough so that $1-\ga +\gep^2/2<0$, and pick $\gl_{10}=\gl_{10}(\gep)$ such that for $\gl\leq \gl_{10}$, $2^{1+\ga}C_5 C_6 (k_{\gl,\gep })^{1-\ga +\gep^2/2}\leq (D_{\gep})^2$.

In the end, for $\gl\leq \min(\gl_8,\gl_9,\gl_{10})$, one has
\begin{equation}
\label{eq:gaplarge}
\sum_{ f_{i-1}\in B_{j_{i-1}}, d_i\in B_{j_i} } K(d_i-f_{i-1}) \leq \frac{ (D_{\gep})^2}{(j_{i}-j_{i-1})^{1+\ga - \gep^2/2}}.
\end{equation}

By inserting \eqref{eq:gapsmall} and \eqref{eq:sumalpha=1}-\eqref{eq:gaplarge} in \eqref{eq:Zcheck4}, one obtains, for $\gl\leq \gl_0(\gep):=\min(\gl_7,\gl_8,\gl_9,\gl_{10})$
\begin{multline}
\label{eq:finalZcheck}
\tilde\bbE_J \big[  \check Z_J \big] \leq (2C_2 C_4)^{\ell} (D_{\gep})^{|\oJ|} (D_{\gep})^{2 |\{i\in\{1,\ldots,\ell\}\, ;\, j_i-j_{i-1} \geq 2\}|}
\prod_{i=1}^\ell  \frac{1}{(j_{i}-j_{i-1})^{1+\ga - \gep^2/2}}\\
 \leq (2C_2 C_4  D_{\gep} )^{\ell}\prod_{i=1}^\ell  \frac{1}{(j_{i}-j_{i-1})^{1+\ga - \gep^2/2}},
\end{multline}
where for the second inequality, we noticed that $D_\gep < 1$ and
\begin{equation}
\ell - |\oJ| = |J\setminus \oJ |\leq 2 |\{i\in\{1,\ldots,\ell\}\, ;\, j_i-j_{i-1} \geq 2\}|.
\end{equation}

\bigskip
{\bf STEP 4: Conclusion of the proof.}
We can now combine \eqref{eq:Holder} with \eqref{eq:part1-chgtmeas} and \eqref{eq:finalZcheck} to estimate $\bbE\big[ (\check Z_J )^{\gz_{\gep}}\big]$, and plug that estimate in \eqref{eq:fracmom} and \eqref{eq:Zhat3}. One ends up, for $\gl\leq \gl_0(\gep)$, with
\begin{equation}
\label{eq:final}
\bbE \big[ \big( Z_{N,\gl, c_{\gep} \gl}^{\omega} \big)^{\zeta_\gep} \big] \leq \sum_{J \subset \{1,\ldots,m\} \, ; \, m\in J} \bigg( \prod_{i=1}^{|J|} \frac{G_{\gep}}{(j_{i}-j_{i-1})^{(1+\ga - \gep^2/2)\zeta_{\gep}} }\bigg),
\end{equation}
where,  recalling the definitions of $C_{\gep}$ and $D_{\gep}$ in \eqref{eq:part1-chgtmeas} and \eqref{defU}, one has
\begin{multline}
G_{\gep}:= 16 (2C_2 C_4  D_{\gep} )^{\zeta_{\gep}} C_{\gep}= 16(2C_2 C_4)^{\zeta_{\gep}} \exp\Big[t_{\gep} \gz_{\gep} \Big\{ (1+\gep^2) \frac{\ga \gU_{\infty}}{2(1+\ga)}  - \Big(\frac{\ga \gU_{\infty}}{2(1+\ga)}  + \gep/2\Big) \Big\} \Big]  \\\leq 32\, C_2 C_4 \exp\Big\{- t_{\gep} \gz_{\gep}\, \frac{ \gep}{2} (1+o(1))\Big\}.
\end{multline}
Here, $o(1)$ is a quantity which goes to zero as $\gep$ goes to zero (containing all the $\gep^2$ terms). If $\gep$ is chosen small enough, then $G_{\gep} \leq 32 C_2 C_4 e^{-t_{\gep} \gep/4(1+\ga)}$.
Therefore, we may choose $t_{\gep}$ large enough to make $G_{\gep}$ arbitrarily small, in particular such that
\begin{equation}
\label{eq:Geps}
\sum_{n=1}^{+\infty} \frac{G_{\gep}}{n^{1+\ga - \gep^2/2}} <1\, .
\end{equation}
Define $\tilde K(n):= G_{\gep}\, n^{-(1+\ga - \gep^2/2)}$ for $n\in\bbN$, which we interpret as the inter-arrival distribution of a transient renewal process, since $\sum_{n\in\bbN} \tilde K (n) <1$. Then, one recognizes in the right hand side of \eqref{eq:final} the probability for this renewal to visit $m$, hence smaller than $1$.

To summarize, we proved that if $\gep$ has been fixed small enough, and $t_{\gep}$ large enough, then, for all $\gl\leq \gl_0(\gep)$, $\liminf_{N\to\infty}\bbE \big[ \big( Z_{N,\gl, c_{\gep} \gl}^{\omega,\cop} \big)^{\zeta_\gep} \big]  \leq 1$. This establishes \eqref{eq:condfracmoment2bis} and concludes the proof.

\subsection{Proof of Theorem \ref{thm:copolrel}}
Here, we assume that correlations $(\rho_{n})_{n\geq1}$ are non-negative, and that $\ga>0$. The annealed critical point is then $h_\a(\gl)=\gU_{\infty}\gl$. The proof also follows a coarse-graining scheme, as that of Toninelli \cite{cf:T1}. We now sketch how to adapt the proof of \cite{cf:T1} to the correlated case, using ideas presented in Section \ref{sec:coarse}.

Let us fix $\zeta\in(\frac{1}{1+\ga}, 1)$, and set $h=h_{\gl}:=\theta \gU_{\infty} \gl$ for some $\theta=\theta(\ga)\in (\zeta,1)$ chosen later in the proof and depending only on $\ga$. We show that, if $\gl$ is small enough, then
\begin{equation}
\liminf_{N\to\infty} \bbE\left[ \big( Z_{N,\gl,h_{\gl}}^{\go,\cop}\big)^{\zeta}\right] <+\infty\, ,
\end{equation}
which enforces Theorem \ref{thm:copolrel}, by the fractional moment argument explained at the beginning of Section \ref{sec:uppercopol}.

We then set up a coarse-graining procedure, with a block length defined as
\begin{equation}
\label{eq:defkFabio}
k = k_{\gl}:= \left\lfloor \frac{1}{\gU_\infty^2\gl^2 (1-\theta)} \right\rfloor,
\end{equation}
and consider a system of length $N:=mk_{\gl}$, in the same way as in Section \ref{sec:coarse}. Step 1 is identical and similarly leads to \eqref{eq:Zhat3}. Note that we use that $\theta >\zeta>1/(1+\ga)$, with the same computations as in \eqref{eq:useMonthusbound}, which is analogous to Equation (3.21) in \cite{cf:T1}. In Step 2, we apply a change of measure with tilting parameter $\gd=1/\sqrt{k_{\gl}}$. The cost of this change of measure, according to \eqref{eq:computecghtmeas} and \eqref{eq:approxgUJJ}, gives a constant $C_7^{|J|}$ ($C_7$ substitutes $C_{\gep}$ in \eqref{eq:part1-chgtmeas}), where $C_7:=\exp \big( \frac12 \frac{\zeta}{1-\zeta} \gU_{\infty}^{-1}\big)$.

We are therefore left with estimating the remaining terms (see \eqref{eq:step2.2}), for which we can apply the decoupling inequality of Lemma \ref{lem:decouple}. Then, using that $h_{\gl}=\theta \gU_{\infty} \gl$, and $\sum_{i,j=1}^n \rho_{ij} \gD_i \gD_j \leq n\gU_{\infty}\leq n \gU_{\infty}$, then for any $n\geq 0$, one has
\begin{equation}
\tilde\bbE_J [z_{0}^{n}] = \frac12 + \frac12 \exp\{-2\gl n (\gl \theta\gU_{\infty} +\gd - \gl \gU_{\infty}) \}\leq 1.
\end{equation}
One therefore ends up with
\begin{multline}
\label{eq:fracmomentFabio}
\bbE\left[ \big( Z_{N,\gl,h_{\gl}}^{\go,\cop}\big)^{\zeta}\right]  \leq \sum_{J\subset \{1,\ldots,m\}} (8\, C_7)^{|J|} \\
\Bigg[  \sumtwo{d_1, f_1 \in B_{j_1}}{d_1 \le f_1}
	\ldots \!\!\!\! \sumtwo{d_{\ell-1}, f_{\ell-1} \in B_{j_{\ell-1}}}{d_{\ell - 1} \le f_{\ell - 1}}
	\sum_{d_\ell \in B_{j_\ell} = B_m}
\prod_{i=1}^\ell 
K(d_i-f_{i-1}) U(f_i-d_i)\Bigg]^{\zeta},
\end{multline}
where $U(\cdot)$ is now defined by $U(n) := \tilde\bbE_J[Z_{n,\gl,h_{\gl}}^{\go,\cop}]=\bbE[Z_{n,\gl,h_{\gl}+1/\sqrt{k}}^{\go,\cop}].$
Note that \eqref{eq:fracmomentFabio} is the analogous to Equations (3.27-3.28) in \cite{cf:T1}, and we are left with showing that $U(n)$ satisfies conditions (3.30)-(3.31) in \cite{cf:T1}.

Using again the non-negativity of correlations and the fact that $h_{\gl}=\theta\gU_{\infty}\gl$, we get
\begin{multline}
U(n) \leq \bE \Big[ e^{(2\gl^2 \gU_\infty (1-\theta) - \frac{2\gl}{\sqrt{k_{\gl}}}) \sum_{i=1}^n \gD_i} \ \gd_n \Big] \\
= \bE \Big[ e^{-2\gl^2 \gU_{\infty} ( 1-\theta - \sqrt{1-\theta}) \sum_{i=1}^n \gD_i}\  \gd_n \Big]  
\leq \bE \Big[e^{-\frac{1}{k_{\gl}} \frac{1}{\gU_{\infty}\sqrt{1-\theta}} \sum_{i=1}^n \gD_i}\ \gd_n \Big],
\end{multline}
where for the second equality, we used the definition \eqref{eq:defkFabio} of $k_{\gl}$, and for the last one, we chose $\theta$ close enough to $1$ so that $2\sqrt{1-\theta} \leq 1$.
This gives the exact same estimate for $U(\cdot)$ as in Equation (3.48) of \cite{cf:T1}, with $\gU_{\infty}\sqrt{1-\theta}$ instead of $\sqrt{1-\theta}$. Then, the proof proceeds identically as in \cite{cf:T1}, since one can make $\frac{1}{\gU_{\infty}\sqrt{1-\theta}}$ arbitrarily large by choosing $\theta$ close to $1$, which was the key to  proving Proposition 3.3 and conditions (3.30)-(3.31) in~\cite{cf:T1}.


\section{Adaptation to the pinning model}
\label{sec:adaptpin}

In this section we sketch how to adapt the techniques of Sections \ref{sec:lowcopol} and \ref{sec:uppercopol} to the pinning model, and prove Theorem \ref{thm:gappin}.

\subsection{Upper bound in the pinning model}
We first focus on proving
\begin{equation}
\label{eq:lower.gappin}
\limsup_{\gb\downarrow 0} \frac{h_{c}^{\pin}(\gb)- h_{\a}^{\pin}(\gb)}{\gb^2}  \leq \frac{\gU_{\infty}}{2\mu} \frac{\ga}{1+\ga} .
\end{equation}
As in Section \ref{sec:lowcopol}, the result comes as a combination of the smoothing inequality in Proposition~\ref{pr:smoothpin} and a lower bound on the free energy, which in this case is given by:
\begin{lemma}\label{lem:lwpin}
For any $c\in\R$,
\begin{equation}
\liminf_{\beta\downarrow 0} \frac{1}{\gb^2}\tf^{\pin}(\beta, c\beta^2) \geq \frac{1}{\mu}\bigg[c+ \frac{1}{2}\bigg(\cpin - \frac{\Upsilon_{\infty}}{\mu} \bigg) \bigg],
\end{equation}
where $\cpin$ has been defined in \eqref{eq:cpin}.
\end{lemma}
Combining Lemma \ref{lem:lwpin} and Proposition~\ref{pr:smoothpin}, as done in \eqref{eq:combine1}-\eqref{eq:combine2} for the copolymer model, gives the right bound, \eqref{eq:lower.gappin}.

\begin{proof}[Proof of Lemma \ref{lem:lwpin}.]
Define
for $N\in \bbN$, $h\in\bbR$ and $\gb\geq 0$ the finite-volume free energies:
\begin{equation}
\tf^{\pin}_N(\gb,h) = \bbE\big[ \frac{1}{N} \log Z^{\go,\pin}_{N,\gb,h}\big], \qquad \tf^{\pin}_{\a,N}(\gb,h) = \frac{1}{N} \log Z^{\pin, \a}_{N,\gb,h},
\end{equation}
where 
\begin{equation}
\label{eq:pin.ann.par.fun}
Z^{\pin, \a}_{N,\gb,h} := \bbE [Z^{\go,\pin}_{N,\gb,h}] = \E \bigg[\exp\bigg\{ h \sum_{n=1}^N \delta_n + \frac{\gb^2}{2} \sum_{n,m=1}^N \rho_{nm} \delta_n \delta_m  \bigg\} \delta_N \bigg].
\end{equation}
The proof relies on the following lemma, which is proven in Appendix \ref{App:interpolation}. 
\begin{lemma}
\label{lem:interpol}
Let $\bP^{\otimes 2}$ be the law of two independent copies of the renewal process, denoted by $\tau$ and $\tau'$. For any $N, M\in\bbN$,
\begin{equation}
\tf^{\pin}(\gb,h)\geq \tf_N^{\pin}(\gb,h)
 \geq \tf^{\pin}_{\a,N}(\gb,h) - \frac{e^{1/M}}{M}  \frac{1}{2N} \log \cA_N + \frac{e^{1/M}}{M} \frac{1}{2N} \log \cB_N \label{eq:lb}
\end{equation}
with
\begin{multline}
\cA_N := \E^{\otimes 2}\bigg[ \exp\bigg\{ (M+1)\gb^2 \sum_{n,m=1}^N \rho_{nm} \delta_n \delta'_m + h \sum_{n=1}^N (\delta_n + \delta'_n) \\
+ \frac{\gb^2}{2} \sum_{n,m=1}^N \rho_{nm} (\delta_n \delta_m + \delta'_n \delta'_m)\bigg\}  \bigg],
\end{multline}
\begin{equation}
\cB_N := \E^{\otimes 2}\bigg[ \exp\bigg\{h \sum_{n=1}^N (\delta_n + \delta'_n) + \frac{\gb^2}{2} \sum_{n,m=1}^N \rho_{nm} (\delta_n \delta_m + \delta'_n \delta'_m)\bigg\}\delta_N \delta'_N \bigg].
\end{equation}
\end{lemma}
The first inequality in \eqref{eq:lb} comes from the super-additivity of $\bbE[\log Z_{N,\gb,h}^{\go}]$, which gives that $\tf^{\pin}(\gb,h) = \sup_{n\in\bbN} \tf^{\pin}_{N}(\gb,h)$. The other inequality in \eqref{eq:lb} is dealt with via interpolation techniques, developed in Appendix \ref{App:interpolation}.
Using Lemma \ref{lem:interpol}, the proof of Lemma \ref{lem:lwpin} consists in giving a second order estimate of \eqref{eq:lb}, for $\gb\searrow 0$, and for appropriate values of $h$ and $N$. Namely, set $h = c\gb^2$ and $N = N_{\gb} := t/\gb^2$, where $t>0$, and apply Lemma \ref{lem:interpol} to get that
\begin{multline}
\label{eq:postinterpol1}
\liminf_{\gb\searrow 0}\frac{1}{\gb ^2} \tf^{\pin}(\gb,c\gb^2) \\
\geq \frac{1}{t} \liminf_{\gb\downarrow0} \log Z_{N_{\gb},\gb,c\gb^2}^{\pin,\a}- \frac{e^{1/M}}{M}  \frac{1}{2t} \limsup_{\gb\downarrow0}\log \cA_{N_{\gb}} + \frac{e^{1/M}}{M} \frac{1}{2t} \liminf_{\gb\downarrow0} \log \cB_{N_{\gb}}.
\end{multline}
Using the convergence results in Lemma \ref{lem:convergence.pin}, and the definition of $N_{\gb}=t/\gb^2$ and $h=c\gb^2$, we get
\begin{equation}
\lim_{\gb\downarrow 0} \, \log \cA_{N_{\gb}} = \frac{t}{\mu} \bigg\{ (M+1)\frac{\gU_{\infty}}{\mu} + 2c + \cpin \bigg\}.
\end{equation}
Also, bounding the $\gd_n$'s by $1$, we get
\begin{multline}
\liminf_{\gb\downarrow 0} \log \cB_{N_{\gb}} \geq 
\liminf_{\gb\downarrow 0}  \bigg\{ \log \P(N_{\gb}\in\tau)^2 -2|c| \gb^2 N_{\gb} - \gb^2 \sum_{n,m=1}^{N_{\gb}} |\rho_{nm}| \bigg\}\\
=- 2\log \mu -2|c| t-t \sum_{n\in\bbZ} |\rho_n|.
\end{multline}
Then, letting $M$ go to infinity in \eqref{eq:postinterpol1}, we obtain that for all $t>0$,
\begin{equation}
\label{eq:Fpin_LB0}
\liminf_{\gb\downarrow 0}\frac{1}{\gb ^2} \tf^{\pin}(\gb,c\gb^2) \geq 
\frac{1}{t}\liminf_{\gb\downarrow0} \log Z_{N_{\gb},\gb,c\gb^2}^{\pin,\a}- \frac{\gU_{\infty}}{2\mu^2}.
\end{equation}
Using Jensen's inequality, we may write
\begin{multline}
\label{eq:Fpinann_LB0}
\log Z_{N_{\gb},\gb,c\gb^2}^{\pin,\a} \geq\\
 \log \P(N_{\gb}\in\tau) + \frac{c\, t}{N_{\gb}}\sum_{ n=1}^{N_{\gb}} \E[\gd_n\mid N_{\gb} \in \tau] + \frac{t}{2 N_{\gb}}\sum_{n,m=1}^{N_{\gb}} \rho_{nm} \E[\gd_n\gd_m \mid N_{\gb}\in\tau],
\end{multline}
and then use the following limits
\begin{equation}
 \lim_{N\to\infty}\frac{1}{N} \sum_{n=1}^N\E[\gd_n\mid N\in \tau] = \frac{1}{\mu},
\quad  \lim_{N\to\infty} \frac{1}{N}\sum_{ n,m=1}^{ N} \rho_{nm} \E[\gd_n\gd_m \mid N \in\tau] = \frac{\cpin}{\mu}.
\end{equation}
The first limit comes from writing $\E[\gd_n\mid N\in \tau]= \bP(n\in\tau) \bP(N-n\in\tau )/ \bP(N\in\tau)$, together with dominated convergence and Cesar\`o summation. A similar reasoning holds for the second limit.
Hence, \eqref{eq:Fpinann_LB0} gives that $\liminf_{\gb\downarrow0} \log Z_{N_{\gb},\gb,c\gb^2}^{\pin,\a} \geq -\log \mu +\frac{t}{\mu} \big(c+\cpin /2 \big)$,
and Lemma \ref{lem:lwpin} follows from \eqref{eq:Fpin_LB0} by letting $t\to\infty$.
\end{proof}

\subsection{Lower bound in the pinning model}
We now turn to the proof of
\begin{equation}
\label{eq:upb.pin}
\liminf_{\gb\downarrow 0} \frac{h_{c}^{\pin}(\gb)- h_{\a}^{\pin}(\gb)}{\gb^2}  \geq \frac{\gU_{\infty}}{2\mu} \frac{\ga}{1+\ga}.
\end{equation}
We use the same techniques as for the copolymer model: we control $\liminf_{N\to\infty} \bbE[(Z^{\go,\pin}_{N,\gb,h})^{\zeta}] $ by gluing finite-size estimates (found as in Section \ref{sec:chgmeas}), with the help of a coarse-graining procedure (as in Section \ref{sec:coarse}). We only focus here on the modifications that are necessary to adapt this scheme to the pinning model.

\subsubsection{Finite-size annealed estimate}
Let us estimate the fractional moment on one block whose length is of order $1/\gb^2$, analogously to Section \ref{sec:chgmeas}.  Recall the expression of the annealed partition function as given in \eqref{eq:pin.ann.par.fun}.
\begin{lemma}[Analogous to Lemma \ref{lem:finitefraccop}]
\label{lem:finitefractpin}
Let $u\in\bbR$, $t>0$ and set 
\begin{equation}
h_{\gb}= \Big(-\frac{\cpin}{2}+u\Big)\gb^2,\qquad k_{\gb}=t/\gb^2.
\end{equation}
Then, for $\zeta\in(0,1)$,
\begin{equation}
\limsup_{\gb\searrow0} \bbE\left[ \left( Z_{k_{\gb},\gb,h_{\gb}}^{\go, \pin} \right)^{\zeta} \right] \leq
	\exp\Big\{\frac{\zeta}{\mu} \bigg( u-   \gU_{\infty}\frac{1-\zeta}{2\mu} \bigg) t \Big\}.
\end{equation}
\end{lemma}
The proof follows exactly the same lines as for Lemma \ref{lem:finitefraccop}, except that we use Lemma \ref{lem:convergence.pin} instead of Lemma \ref{lem:convergenceapp} (in particular in \eqref{eq:modifmeascop}). Details are left to the reader.

Recall Proposition \ref{prop:pinann}. According to Lemma \ref{lem:finitefractpin}, the fractional moment of the partition function should therefore stay bounded when $h \leq \gb^2 \big(-\frac{\cpin}{2} + \frac{\gU_{\infty}(1-\zeta)}{2\mu} \big)$. Since we need $\zeta > \frac{1}{1+\ga}$ to make the full coarse-graining procedure work (see Section \ref{sec:coarse}), we get \eqref{eq:upb.pin}.

\subsubsection{Coarse-graining procedure.} We only stress the main modifications of the steps in Section \ref{sec:coarse} that are needed to adapt the proof to the pinning model.

\medskip \noindent
{\bf STEP 0.} We set
\begin{equation}
\label{eq:choice.c.pin}
c_0 := - \frac{\cpin}{2} + \frac{\gU_\infty}{2\mu} \frac{\ga}{1+\ga},\quad c_\gep = c_0 - \gep,\quad \gep >0.
\end{equation}
We need to show that there exists $\gl = \gl_0(\gep)$ and $\zeta_\gep\in (0,1)$ such that for $\gl \in (0,\gl_0)$,
\begin{equation}
\liminf_{N\to\infty} \bbE[(Z^{\go,\pin}_{N,\gb,c_\gep\beta^2})^{\zeta_\gep}] < +\infty.
\end{equation}
Similarly to \eqref{eq:choice.par.cop}, we write $h=h_{\gep}(\gb)=c_{\gep} \gb^2$, and set $\zeta_{\gep}=\frac{1}{1+\ga} +\gep^2$, $k_{\gb,\gep}=t_{\gep}/\gb^2$, where $t_{\gep}$ is chosen large enough at the end of the proof.

\medskip \noindent
{\bf STEP 1.} The  coarse-graining decomposition in \eqref{eq:1stdecomp} is still valid, with \eqref{eq:Zdf} and \eqref{eq:z.onestretch} replaced by 
\begin{equation}
a\leq b,\quad z_a^b = z^b = e^{\gb \go_b + h},\qquad Z_{a,b} = Z^{\gt^a\go,\pin}_{b-a,\gb,k_{\gb,\gep}}.
\end{equation}
The steps (\ref{eq:smallz.dec}--\ref{eq:Zhat3}) are no longer necessary, and one is left with estimating $\bbE[(\hat{Z}_J)^{\zeta_\gep}]$.

\medskip \noindent
{\bf STEP 2.} We use the same type of change of measure procedure. Here, the law $\tilde{\p}_J$ has a tilting parameter
\begin{equation}
\label{eq:choice.tilt.pin}
\gd = a_\gep \gb, \qquad \mbox{where}\quad a_\gep = -(1-\zeta_\gep)\gU_\infty/\mu.
\end{equation}
Equation \eqref{eq:part1-chgtmeas} stills holds, with
\begin{equation}
C_\gep = \exp\Big( \zeta_\gep(1+\gep^2)\frac{t_\gep}{\mu^2}\frac{\gU_\infty}{2}\frac{\ga}{1+\ga} \Big).
\end{equation}
The quantity to estimate is now, analogously to \eqref{eq:step2.2}, 
\begin{equation}
\tilde{\bbE}_J\Big[ \prod_{i=1}^\ell Z_{d_i,f_i} z^{d_i} \Big] 
\leq 2^\ell \prod_{i=1}^\ell \tilde{\bbE}_J [  Z_{d_i,f_i} ] \tilde{\bbE}_J[z^{d_i} ]
\end{equation}
where the inequality holds for $\gb$ small enough, using decoupling inequalities (see Lemma \ref{lem:decouple.general}, and Remark \ref{rmk:dec.pin}). Since $\tilde{\bbE}_J[z^{d_i} ]\leq 2$ if $\gb$ is small enough, we focus on $\tilde{\bbE}_J [  Z_{d_i,f_i} ] $. One has
\begin{equation}
\tilde{\bbE}_J[Z_{0,n}] = \bE\Big[ \exp\Big( (c_\gep + a_\gep) \gb^2 \sum_{i=1}^n \gd_i + \frac{\gb^2}{2} \sum_{i,j=1}^n \rho_{ij} \gd_i \gd_j \Big) \delta_n \Big],\quad n\in \{1,\ldots,k_{\gb,\gep}\}.
\end{equation}

Using the convergence in Lemma \ref{lem:convergence.pin} together with the definition of $k_{\gb,\gep}=t_{\gep}/\gb^2$, one can find $\gb_1=\gb_1(\gep)$ such that for $\gb\leq \gb_1$, one gets that for any $n\in \{(1-\gep^2) k_{\gb,\gep},\ldots,k_{\gb,\gep}\}$, 
\begin{equation}
\tilde{\bbE}_J[Z_{0,n}] \leq \exp\Big\{\big(c_\gep + a_\gep\big) (1-\gep^2) \frac{t_\gep}{\mu} + \frac{t_\gep}{\mu} \frac{\cpin}{2}\Big\} \leq \exp\Big\{-\frac{t_\gep}{\mu}\Big(\frac{\gU_\infty}{2\mu}\frac{\ga}{1+\ga} + \frac{\gep}{2}\Big)\Big\},
\end{equation}
where for the second inequality, we used the definitions of $c_{\gep}$ and $a_{\gep}$, and chose $\gep$ small.

Provided that $\gb$ is small enough, one also has a uniform bound $\tilde\bbE_J[ Z_{0,n}]\leq C_1$ for $n\in\{1,\ldots,k_{\gb,\gep}\}$, where the constant $C_1$ does not depend on $\gep$, as in STEP 2.\ (2.b.ii) (the analogous to  \eqref{eq:STEP2bii.one}-\eqref{eq:STEP2bii.two} hold here thanks to Lemma \ref{lem:convergence.pin}).

In the end, we obtain the analogous to \eqref{eq:Zcheck2} and \eqref{defU}, where the constant $D_\gep$ becomes
\begin{equation}
D_\gep = \exp\Big( -\frac{t_\gep}{\mu} \Big( \frac{\gU_\infty}{2\mu}\frac{\ga}{1+\ga} + \frac{\gep}{2} \Big) \Big).
\end{equation}


\medskip \noindent
{\bf STEP 3} is identical to the copolymer case, and one ends up in {\bf STEP 4} with the analogous to \eqref{eq:final}, with 
\begin{multline}
G_{\gep}= C_8 ( D_{\gep})^{\zeta_{\gep}} C_{\gep} 
= C_8 \exp \Big\{-\frac{t_{\gep} }{\mu}\gz_{\gep} \Big( \frac{\gU_\infty}{2\mu}\frac{\ga}{1+\ga} + \frac{\gep}{2} \Big) +\zeta_\gep(1+\gep^2)\frac{t_\gep}{\mu^2}\frac{\gU_\infty}{2}\frac{\ga}{1+\ga}  \Big\} \\
 \leq C_8 \exp\Big(-\frac{t_{\gep} }{\mu}\gz_{\gep} \frac{\gep}{4}\Big),
\end{multline}
where (i) $C_8$ is a constant which does not depend on $\gep$ and (ii) we took $\gep$ small enough so that all the terms on order $\gep^2$ become negligible. Then, one can make $G_{\gep}$ arbitrarily small by choosing $t_{\gep}$ large, so that \eqref{eq:Geps} holds. This concludes the proof. \qed


\begin{appendix}

\section{Convergence results}
\label{app1}

\begin{lemma}[Convergence in the copolymer model]
\label{lem:convergenceapp}
If $\mu=\bE[\tau_1]<+\infty$, then $\bP$-a.s.,
\begin{align}
\lim_{N\to\infty}\frac{1}{N} \sum_{n=1}^N \gD_n &= \frac12,\label{conv:cascopol0}\\
\lim_{N\to\infty}\frac{1}{N}\sum_{n,m=1}^{N}  \rho_{nm} \gD_n \gD_m &= \lim_{N\to\infty}\frac{1}{N}\sum_{n,m=1}^{N}  \rho_{nm} \bE[\gD_n \gD_m] =
 \frac{1}{4} \gU_{\infty} +\frac{1}{4} \ccop, \label{conv:cascopol1}
\end{align}
where $\ccop$ is defined in \eqref{defccop}.
\end{lemma}

\begin{lemma}[Convergence in the pinning model]
\label{lem:convergence.pin}
Let $\bP^{\otimes 2}$ refer to the law of two independent copies of the renewal process, denoted by $\tau$ and $\tau'$.
If $\mu=\bE[\tau_1]<+\infty$, then $\bP^{\otimes 2}$-a.s,
\begin{align}
\lim_{N\to\infty}\frac{1}{N}\sum_{n=1}^{N} \gd_n &= \frac{1}{\mu}, \label{conv:caspin0}\\
\lim_{N\to\infty}\frac{1}{N}\sum_{n,m=1}^{N}  \rho_{nm} \gd_n \gd_m &= \lim_{N\to\infty}\frac{1}{N}\sum_{n,m=1}^{N}  \rho_{nm} \bE[\gd_n \gd_m] =
 \frac{\cpin}{\mu},\label{conv:caspin1}\\
\lim_{N\to\infty}\frac{1}{N}\sum_{n,m=1}^{N}  \rho_{nm} \gd_n \gd'_m &= \lim_{N\to\infty}\frac{1}{N}\sum_{n,m=1}^{N}  \rho_{nm} \bE^{\otimes 2}[\gd_n \gd'_m] =
 \frac{\gU_\infty}{\mu^2}, \label{conv:caspin2}
\end{align}
where $\cpin$ has been defined in \eqref{eq:cpin}.
\end{lemma}

\begin{proof}[Proof of Lemmas \ref{lem:convergenceapp} and \ref{lem:convergence.pin}.]
We borrow here an idea that can be used to prove the Renewal Theorem, see \cite{cf:Asm}.
We recall the definition of the \emph{backward recurrence time process} $(A_k)_{k\in \bbN_0}$. It is a Markov chain which indicates the
time elapsed since the last renewal: $A_k = k - \tau_{\cN_k}$, where $\cN_k = |\tau\cap\{1,\ldots, k\}|$.
We can decorate the Markov chain $(A_k)_{k\in\bbN_0}$ by adding the ``sign'' $\Delta_k$, defining $\tilde A_k:=(A_k, \Delta_k)$
which is also a Markov chain.
Since $K(n)>0$ for all large $n\in \bbN$ (see \eqref{defK}) and $K(\infty)=0$,
$(\tilde A_k)_{k\geq 0}$
is recurrent, irreducible, and aperiodic, and  with the additional assumption that $\mu = \bE[\tau_1] <\infty$, the stationary probability measure $\pi$ is explicit:
$\pi(a,\Delta) = \frac{1}{2\mu} \bP(\tau_1\geq a)$.

Let us now define a generalization of this process, which memorizes the last $q$ states, for some given integer $q$ (arbitrarily large):
let $W_{k}:=(\tilde A_k, \tilde A_{k-1}, \ldots, \tilde A_{k-q})$, where by convention
we set $\tilde A_k=(0,0)$ for $k<0$.
Then, $(W_k)_{k\in \bbN_0}$ is also a  positive recurrent irreducible aperiodic Markov chain, with an explicit stationary probability measure denoted by $\Pi$. If we write $W=((a_0,\Delta_0),\ldots,(a_q,\Delta_q))$ and denote by $j_1<j_2<\cdots<j_m$ the (ordered) indices such that $a_{j_k}=0$ (the renewal points), then
\begin{equation}
\Pi\big( W \big)
    = \pi(a_0,\Delta_0) \prod_{k=1}^{m-1} \frac12 K(j_{k+1}-j_{k})
  \prod_{j\notin \{j_1,\ldots,j_m\} }\!\!\! \ind_{\{ a_{j+1}=a_{j}-1\}} \ind_{\{\Delta_{j+1} = \Delta_{j}\}}.
  \end{equation}

By using the ergodic theorem for Markov chains,
we get that for any bounded function $G: \bbN^{2q} \to \bbR$,
$\frac1N \sum_{k=1}^{N} G(U_k)$ converges $\P$-a.s.\ to $\bE_{\Pi}[G(U_0)]$ as $N\to\infty$.
Applying this result to the test functions $G(U)=\ind_{\{a_0=0\}}$ and $G(U)= \Delta_0$, one gets (\ref{conv:cascopol0}) and (\ref{conv:caspin0}). To prove \eqref{conv:caspin1}, we write
\begin{equation}
 \frac{1}{N}\sum_{n,m=1}^{N}  \rho_{nm} \gd_n \gd_m  = \frac{1}{N} \sum_{n=1}^{N}   \gd_n
    + \frac{2}{N} \sum_{n=1}^{N}   \gd_n \sum_{j=1}^q \rho_{j} \gd_{ n+j} + R_N, 
\end{equation}
where the error term $R_N$ satisfies $|R_N| \leq 2 \sum_{j\geq q} |\rho_j| + \frac{q}{N} \sum_{j\geq1} |\rho_j|$.
Taking the limit $N\to\infty$, and then letting $q\to \infty$,
we obtain \eqref{conv:caspin1}, since 
\begin{equation}
\bE_{\Pi}\Big[ \ind_{\{a_0=0\}} \sum_{j=1}^q \rho_j \ind_{\{a_j=0\}}\Big]
 = \frac{1}{\mu}\sum_{j=1}^q \rho_j \bP(j\in\tau).
 \end{equation}

We proceed in the same way to obtain \eqref{conv:cascopol1}. Compute first
\begin{multline}
  \bE_{\Pi}\Big[ \gD_0 \sum_{j=1}^q \rho_j \gD_j\Big]= \frac{1}{4}\sum_{j=1}^q \rho_j +  \frac{1}{4} \sum_{n=1}^q \Big(\sum_{j=1}^{n} \rho_j \Big)\frac{1}{\mu}\bP(\tau_1 \geq n +1) \\
  + \frac14 \big(\sum_{j=1}^{q} \rho_j \big) \sum_{k\geq q+1} \frac{1}{\mu} \bP(\tau_1\geq k+1).
\end{multline}
Then, one can write $\frac{1}{N}\sum_{n,m=1}^{N}  \rho_{nm} \gD_n \gD_m  = \frac{1}{N} \sum_{n=1}^{N}   \gD_n
    + \frac{2}{N} \sum_{n=1}^{N}   \gD_n \sum_{j=1}^q \rho_{j} \gD_{ n+j} + \tilde R_N$, where, similarly to the above, one has $\lim_{q\to\infty}\lim_{N\to\infty}\tilde R_N = 0$.

Therefore, we have $\bP$-a.s.
\begin{multline}
\lim_{N\to\infty} \frac{1}{N}\sum_{n,m=1}^{N}  \rho_{nm} \gD_n \gD_m = \frac12 +2 \lim_{q\to\infty}  \bE_{\Pi}\Big[ \gD_0 \sum_{j=1}^q \rho_j \gD_j\Big]\\
=\frac14 +\frac14 \gU_{\infty} +\frac{1}{2\mu} \sum_{n\in\bbN} \bbP(\tau_1\geq n+1) \Big(\sum_{j=1}^{n} \rho_j\Big).
\end{multline}
This actually gives exactly \eqref{conv:cascopol1}, because
\begin{multline}
\ccop = \frac{1}{\mu}\bE\bigg[ \sum_{i,j=1}^{\tau_1} \rho_{ij}\bigg]
 = 1+ \frac{2}{\mu}\bE\bigg[ \sum_{i=1}^{\tau_1} \sum_{k=1}^{i-1} \rho_k\bigg]\\
 = 1+ \frac{2}{\mu} \sum_{n\in\bbN} \bP(\tau_1=n) \sum_{i=1}^{n} \sum_{k=1}^{i-1} \rho_k
 = 1 + \frac{2}{\mu} \sum_{n\in\bbN} \bP(\tau_1\geq n +1) \sum_{k=1}^{n} \rho_k,
\end{multline}
where the last inequality is obtained via an Abel summation by parts.

The convergence in \eqref{conv:caspin2}  is obtained by the same lines of reasoning, using two independent copies of the backward recurrence time process. Details are left to the reader.
\end{proof}

\begin{rem}\rm
\label{rem:muinfty}
We want to stress that if $\mu=+\infty$, then the following convergence holds:
\begin{equation}
\label{conv:cascopol2}
 \lim_{N\to\infty}\frac{1}{N}\sum_{n,m=1}^{N}  \rho_{nm} \bE[\gD_n \gD_m] =
\frac12 \gU_{\infty},
\end{equation}
suggesting the interpretation $\ccop =\gU_{\infty}$ in that case. The limit in \eqref{conv:cascopol2} can be proved by observing that $\bE[\gD_n \gD_m] = \frac{1}{2} - \frac14 \bP(\tau\cap [n,m) \neq \emptyset)$. A union bound gives that, for any fixed $p\geq 1$, $\bP(\tau\cap [n,n+p) \neq \emptyset) \leq \sum_{k=n}^{n+p} \bP(k\in\tau) \stackrel{n\to\infty}{\to} 0$, because $\mu=+\infty$ implies that $\bP(k\in\tau)$ converges to $0$. Since $\sum_{k\in\bbZ} |\rho_k| <+\infty$, we get by dominated convergence that 
\begin{equation}
\lim_{n\to\infty} \sum_{m=n}^{+\infty}  \rho_{nm} \bP(\tau\cap [n,m) \neq\emptyset)= \lim_{n\to\infty} \sum_{k\in\bbN} \rho_k \bP(\tau \cap [n,n+k) \neq \emptyset) = 0.
\end{equation}
Using Cesar\`o summation together with the symmetry in $n,m$, one gets \eqref{conv:cascopol2}.
\end{rem}

\section{Decoupling inequalities}
\label{sec:decouple.app}
We first state a general decoupling inequality, that we prove using standard interpolation techniques, and then explain how to apply it to obtain Lemmas \ref{lem:decoupling0} and \ref{lem:decouple}.

Let us first recall the Gaussian integration-by-part formula. If $\go = (\go_n)_{n\in\bbZ}$ is a stationary Gaussian sequence with correlation matrix $\gU = (\rho_{ij})_{i,j\in\bbZ}$, and $f$ a sufficiently smooth function, then
\begin{equation}
\label{eq:Gaussian.IPP}
\forall k\in\bbZ,\qquad \bbE[\omega_k f(\go)] = \sum_{n\in\bbZ} \rho_{kn} \bbE[\partial_{\go_n}f(\go)].
\end{equation}

\begin{lemma}[{\bf General Decoupling Inequality}]
\label{lem:decouple.general}
Let $\bbP$ be the law of a centered Gaussian sequence $\go=(\go_n)_{n\in\bbZ}$, with a correlation matrix $\gU= (\rho_{ij})_{i,j\in\bbZ}$. If $\cI$ and $\cJ$ are two disjoint subsets of $\bbZ$, we define
\begin{equation}
C(\cI,\cJ) = \sum_{i\in \cI,\, j\in\cJ} \mid \rho_{ij} \mid.
\end{equation}
Let $f:\bbR^\cI \mapsto \bbR$ and $g:\bbR^\cJ \mapsto \bbR$ be two functions which are $\cC^2$ and such that, for some constant $c\in(0,\infty)$,
\begin{equation}
\label{eq:ass.f.g}
\begin{split}
\forall i\in \cI,\, j\in \cJ,\qquad & \mid \partial_{\omega_i}f \mid  \leq c\, f,\qquad \mid \partial_{\omega_j}g \mid  \leq c\, g \, ;\\
\forall i,i'\in \cI,\, j,j'\in \cJ, \qquad & \mid \partial^2_{\omega_i,\go_{i'}}f \mid  \leq c^2\, f,\qquad \mid \partial^2_{\omega_j,\go_{j'}}g \mid  \leq c^2\, g \, .
\end{split}
\end{equation}
Denoting by $\go_\cI$ the vector $(\go_n)_{n\in \cI}$, for any subset $\cI$ of $\bbZ$, we have
\begin{equation}
\bbE[f(\go_\cI) g(\go_\cJ)] \leq e^{c^2 C(\cI,\cJ)} \bbE[f(\go_\cI)] \bbE[g(\go_\cJ)].
\end{equation}
\end{lemma}

\begin{proof}
By using two independent copies of $\go$, one can build a centered Gaussian sequence $\tilde{\go}$, independent of $\go$, with the following covariance structure:
\begin{equation}
\label{eq:cov.gto}
\forall i,j\in \cI\cup \cJ,\qquad \bbE[\gto_i \gto_j] = \rho_{ij} (\ind_{\{i,j\in \cI\}} + \ind_{\{i,j\in \cJ\}}).
\end{equation}
For convenience we still use the symbol $\bbP$ for the law of $(\go,\gto)$. We interpolate between $\go$ and $\gto$ by defining
\begin{equation}
\forall n\in\bbZ,\,\forall t\in[0,1],\qquad \go_n(t) = \sqrt{t}\, \gto_n + \sqrt{1-t}\, \go_n,
\mbox{ and } \quad \varphi(t) = \bbE[f(\go_\cI(t)) g(\go_\cJ(t))].
\end{equation}
Note that $\varphi(0) = \bbE[f(\go_\cI) g(\go_\cJ)]$ and $\varphi(1) = \bbE[f(\go_\cI)] \bbE[g(\go_\cJ)]$, so what we want to prove is
\begin{equation}
\varphi(0) \leq C \varphi(1), \qquad \mbox{where }\quad C = e^{c^2 C(\cI,\cJ)}.
\end{equation}
By Gronwall's lemma, it is enough to prove the following inequality:
\begin{equation}
\label{eq:decouple.Gronwall}
\forall t\in (0,1),\qquad \varphi'(t) \geq - C \varphi(t).
\end{equation}
Let us first compute $\varphi'(t)$:
\begin{multline}
\varphi'(t) = \bbE[\partial_tf(\go_\cI(t))\, g(\go_\cJ(t))] + \bbE[f(\go_\cI(t))\, \partial_tg(\go_\cJ(t))]\\
= \sum_{i\in \cI} \bbE[\partial_t \go_i(t)\, \partial_{\go_i}f(\go_\cI(t))\, g(\go_\cJ(t))] + \sum_{j\in \cJ} \bbE[\partial_t \go_j(t)\, f(\go_\cI(t))\, \partial_{\go_j}g(\go_\cJ(t))],
\end{multline}
and since
\begin{equation}
\forall n\in\bbZ,\qquad \partial_t \go_n(t) = \frac{1}{2\sqrt{t}}\, \gto_n - \frac{1}{2\sqrt{1-t}}\, \go_n,
\end{equation}
we obtain
\begin{multline}
\varphi'(t) = \frac12 \sum_{i\in \cI} \bbE\Big[\Big(\frac{\gto_i}{\sqrt{t}} - \frac{\go_i}{\sqrt{1-t}}\Big) \partial_{\go_i}f(\go_\cI(t))\, g(\go_\cJ(t))\Big]\\
+ \frac12 \sum_{j\in \cJ} \bbE\Big[\Big(\frac{\gto_j}{\sqrt{t}} - \frac{\go_j}{\sqrt{1-t}}\Big) f(\go_\cI(t))\, \partial_{\go_j} g(\go_\cJ(t))\Big].
\end{multline}
Using \eqref{eq:Gaussian.IPP} and \eqref{eq:cov.gto}, we obtain for $i\in \cI$,
\begin{equation}
\bbE[\gto_i\, \partial_{\go_i}f(\go_\cI(t))\, g(\go_\cJ(t))] = \sqrt{t} \sum_{j\in \cI} \rho_{ij}  \bbE[\partial^2_{\go_i,\go_j}f(\go_\cI(t))\, g(\go_\cJ(t))],,
\end{equation}
\begin{multline}
\bbE[\go_i\, \partial_{\go_i}f(\go_\cI(t))\, g(\go_\cJ(t))] = \sqrt{1-t} \sum_{j\in \cI} \rho_{ij}  \bbE[\partial^2_{\go_i,\go_j}f(\go_\cI(t))\, g(\go_\cJ(t))]\\
+ \sqrt{1-t} \sum_{j\in \cJ} \rho_{ij} \bbE[\partial_{\go_i}f(\go_\cI(t))\, \partial_{\go_j}g(\go_\cJ(t))],
\end{multline}
and for $j\in \cJ$,
\begin{equation}
\bbE[\gto_j f(\go_\cI(t))\, \partial_{\go_j}g(\go_\cJ(t))] = \sqrt{t} \sum_{i\in \cJ} \rho_{ij}  \bbE[f(\go_\cI(t))\, \partial^2_{\go_i,\go_j} g(\go_\cJ(t))],
\end{equation}
\begin{multline}
\bbE[\go_j\, f(\go_\cI(t))\, \partial_{\go_j} g(\go_\cJ(t))] = \sqrt{1-t} \sum_{i\in \cJ} \rho_{ij}  \bbE[f(\go_\cI(t))\, \partial^2_{\go_i,\go_j}g(\go_\cJ(t))]\\
+ \sqrt{1-t} \sum_{i\in \cI} \rho_{ij} \bbE[\partial_{\go_i}f(\go_\cI(t))\, \partial_{\go_j}g(\go_\cJ(t))].
\end{multline}
Adding everything up, we get
\begin{equation}
\varphi'(t) = -\sum_{i\in \cI,\, j\in\cJ} \rho_{ij} \bbE[\partial_{\go_i}f(\go_\cI(t)) \partial_{\go_j}g(\go_\cJ(t))],
\end{equation}
from which we deduce \eqref{eq:decouple.Gronwall}, thanks to \eqref{eq:ass.f.g}.
\end{proof}

\smallskip
\begin{proof}[Proof of Lemmas \ref{lem:decoupling0} and \ref{lem:decouple}.] We apply iteratively Lemma \ref{lem:decouple.general} to specific functions of $\go$.
Observe that, in any decoupling lemma that we are using (Lemmas \ref{lem:decoupling0} and \ref{lem:decouple}), the functions to which we apply Lemma \ref{lem:decouple.general} are of the form $\cZ(\go)^{\gamma}$, where $0<\gamma\leq 1$ and $\cZ(\go)$ is a finite positive linear combination of finite products of functions of the form
\begin{equation}
Z_{\cI_r}(\go) = \bE\big[e^{-2\gl \sum_{n\in \cI_r} (\go_n + h)\sigma_n}\big],
\end{equation}
for disjoint intervals $\cI_r$, and where $\sigma_n$ is a random variable taking values in $\{0,1\}$.
Note that the choice $\sigma_n=\gd_n$ gives the partition function of the pinning model (up to a change of parameters); $\sigma_n = \gD_n$ the partition function of the copolymer model; and finally $\bP(\sigma_n = 1,\, \forall n\in I_k) = \bP(\sigma_n = 0,\, \forall n\in I_k) = 1/2$ leads to the partition function restricted to one large excursion of size $|\cI_r|$, as in \eqref{eq:z.onestretch}.

It is then easy to check that for all $i,i'\in \cI_r$, $\mid\!\! \partial_{\go_i}Z_{\cI_r}(\go)\!\! \mid \leq 2\gl Z_{\cI_r}(\go)$, and $\mid\!\! \partial^2_{\go_i, \go_{i'}}Z_{\cI_r}(\go) \!\!\mid \leq 4\gl^2 Z_{\cI_r}(\go)$. Therefore, for $i,i'\in\bigcup_r \cI_r$,
\begin{equation}
\mid\!\! \partial_{\go_i} \cZ(\go)^{\gamma} \!\! \mid = \gamma \mid\!\!  [\partial_{\go_i} \cZ(\go)] \cZ(\go)^{\gamma-1}\!\!  \mid  \leq 2\gl\, \cZ(\go)^{\gamma}, \quad \text{and} \quad \mid \!\! \partial^2_{\go_i,\go_{i'}} \cZ(\go)^{\gamma} \!\! \mid \leq 4\gl^2\,  \cZ(\go)^{\gamma},
\end{equation}
which proves that \eqref{eq:ass.f.g} is fulfilled in our context, with a constant $c=2\lambda$.

Now, we can use Lemma \ref{lem:decouple.general}, and only the different constants $C(\cI,\cJ)$ involved remain to be estimated. Note that $C(\cI,\cJ)\leq \sum_{i\in\cI}\sum_{j\notin\cI} |\rho_{ij}|$, so that the definition of $\cJ$ does not matter. To obtain Lemma \ref{lem:decoupling0}, we apply Lemma \ref{lem:decouple.general} only once, with $\cI = \bigcup_{i=1}^{\ell} B_{j_i}$ (as defined in Section \ref{sec:coarse}): one gets $C(\cI,\cJ)\leq \ell V_{k_{\gl,\gep}}$, where $V_k :=2\sum_{i=1}^k \sum_{j > k} |\rho_{ij}|$. To obtain Lemma \ref{lem:decouple}, we apply  Lemma \ref{lem:decouple.general} repeatedly ($3\ell-1$ times), with each time some constant $C(\cI,\cJ)$ where $\cI$ is an interval of length $s\leq k_{\gl,\gep}$, so that $C(\cI,\cJ)\leq V_s$.
For both lemmas, we are therefore only left to show that if $\lambda$ is small enough, then $4\gl^2 V_s \leq \log 2$, for all $s\in\{1,\ldots,k_{\gl,\gep}\}$.

Indeed, $V_s = \epsilon(s) \, s$, where $\epsilon (s)$ goes to $0$ as $s$ goes to infinity, because of the summability of the correlations. One controls $4\lambda^2 V(s) = 4 t_{\gep} \epsilon(s) s/k_{\gl,\gep}$. Let us write $s_0:=(4t_{\gep} \max_{s\in\bbN}\{\epsilon (s)\})^{-1} \log 2$: if $s\leq s_0\, k_{\gl,\gep}$, then  $4\lambda^2 V(s) \leq \log 2$ ; if $s\in  (s_0\,k_{\gl,\gep}, k_{\gl,\gep})$, then $4\lambda^2 V(s) \leq 4t_{\gep} \epsilon (s)$, which can be made arbitrarily small by taking $s\geq s_0 k_{\gl,\gep}$ large, in other words by taking $\lambda$ small.
\end{proof}

\begin{rem}\rm
\label{rmk:dec.pin}
For the sake of conciseness, we do not detail the case of the pinning model, and leave it to the reader to check that analogous to Lemmas \ref{lem:decoupling0}-\ref{lem:decouple} hold, with an almost identical proof. Let us point out that the case of the pinning model is even simpler since the polymer does not collect any charge during the large excursions that occur between the visits of coarse-grained blocks.
\end{rem}

\section{Proof of Lemma \ref{lem:interpol}}
\label{App:interpolation}

The interpolation techniques we use here are inspired from \cite{cf:T2}, but many adaptations are however needed to deal with the correlated case. One writes
\begin{align}
&\tf^{\pin}_N(\gb,h) =\nonumber\\
&\hspace{0.5cm} \tf^{\pin}_{\a,N}(\gb,h) + \bbE \frac{1}{N} \log \E_{N,
\gb,h} \Big[ \exp\Big\{\gb \sum_{n=1}^N \omega_n \delta_n - \frac{\gb^2}{2} \sum_{n,m=1}^N \rho_{nm} \delta_n \delta_m \Big\} \Big],
\end{align}
where
\begin{equation}
\label{eq:defEngdh}
\E_{N,\gb,h}[\cdot] = \frac{\E\left[\,\cdot\,\exp\left\{ h \sum_{ n=1}^{ N} \delta_n + \frac12 \gb^2 \sum_{n,m=1}^{N} \rho_{nm} \delta_n \delta_m  \right\} \delta_N \right]}{\E\left[\exp\left\{h \sum_{n=1}^{N} \delta_n + \frac12 \gb^2 \sum_{n,m=1}^{N} \rho_{nm} \delta_n \delta_m  \right\} \delta_N \right]}.
\end{equation}
Let us denote by $\tau'$ an independent copy of  $\tau$. We define for $t\in[0,1]$, which plays the role of an interpolation parameter, and for $\kappa\in \bbR$, which can be seen as a coupling constant:
\begin{align}
\label{eq:defHngbtkappa}
H^{\otimes 2}_N(\gb,t,\kappa,\omega) := \sqrt{t}\gb \sum_{n=1}^N \omega_n (\delta_n + \delta_n') &- \frac{t\gb^2}{2} \sum_{n,m = 1}^N \rho_{nm} (\delta_n \delta_m + \delta_n' \delta_m')\nonumber\\& + \kappa \gb^2 \sum_{n,m = 1}^N \rho_{nm} \delta_n \delta_m',  
\end{align}
and for $\beta\geq 0$,
\begin{equation}
\psi(N,h\gb;t,\kappa) := \bbE \frac{1}{2N} \log \E^{\otimes 2}_{N,\beta,h}\bigg[\exp\{ H^{\otimes 2}_N(\gb,t,\kappa,\omega)\}\bigg].
\end{equation}

Since
\begin{align}
&\psi(N,h,\gb;t,\kappa=0) = \nonumber \\
&\hspace{2cm} \bbE \frac{1}{N} \log \E_{N,\beta,h}\Big[ \exp\Big\{\sqrt{t}\gb \sum_{n=1}^N \omega_n \delta_n - \frac{t\gb^2}{2} \sum_{n,m=1}^N \rho_{nm} \delta_n \delta_m \Big\}\Big],
\end{align}
we have
\begin{equation}
\label{eq:manquegronwall}
\tf^{\pin}(\gb,h)\geq \tf^{\pin}_N(\gb,h) = \tf^{\pin}_{\a,N}(\gb,h) + \psi(N,h,\gb;t=1,\kappa=0).
\end{equation}
To control $\psi(N,h,\gb;t=1,\kappa=0)$, we use
\begin{gather}
\label{eq:interpolation}
\partiald{t}\psi(N,h,\gb;t,\kappa) \leq  \partiald{\kappa}\psi(N,h,\gb;t,\kappa),\\
- \partiald{t}\psi(N,h,\gb;t,\kappa=0) = \partiald{\kappa}\psi(N,h,\gb;t,\kappa=0).
\label{eq:equ_partiald}
\end{gather}
The proofs of \eqref{eq:manquegronwall} and \eqref{eq:interpolation} are postponed to the end of this section.

Then, since for all $N,h,\gb,t$, the map $\kappa \to \psi(N,h;\gb,t,\kappa)$ is convex and non-decreasing on $\bbR^+$ (as it is clear by computing the first and second order derivatives with respect to $\kappa$, see \eqref{eq:derivkappa}), one gets for all $M>0$ and $t\in[0,1]$,
\begin{align}
\partiald{\kappa}\psi(N,h,\gb;t,\kappa=0) &\leq \frac{\psi(N,h,\gb;t,1+M-t) - \psi(N,h,\gb;t,0)}{1+M-t}\nonumber\\
&\leq \frac{\psi(N,h,\gb;0,1+M) - \psi(N,h,\gb;t,0)}{1+M-t}\nonumber\\
&\leq \frac{\psi(N,h,\gb;0,1+M) - \psi(N,h,\gb;t,0)}{M},\label{eq:useconvexity1}
\end{align}
where for the second inequality, we used that for $x\geq 0$, $\psi(N,h,\gb;x,\kappa) \leq \psi(N,h,\gb;0,\kappa+x)$, which is a consequence of \eqref{eq:interpolation}.
Therefore, setting 
\begin{equation}
f(t) := - \psi(N,h;\gb,t,\kappa=0),\qquad
c(M) := \psi(N,h;\gb,t=0,\kappa=M+1),
\end{equation}
and using \eqref{eq:equ_partiald} and \eqref{eq:useconvexity1}, we get $\partiald{t}f(t) = \partiald{\kappa}\psi(N,h;\gb,t,\kappa=0)\leq \frac{c(M)}{M} + \frac{f(t)}{M}$, so that the derivative of $f(t)e^{-t/M}$ is uniformly bounded by $c(M)/M$.
Since $f(0)=0$, we get by integrating between $t=0$ and $t=1$ that
$f(1) \leq \frac{c(M)}{M} e^{1/M}$, that is
\begin{equation}
\label{eq:gronwall}
\psi(N,h,\gb;t=1,\kappa=0) \geq - \frac{e^{1/M}}{M} \psi(N,h;\gb,t=0,1+M).
\end{equation}
The combination of \eqref{eq:manquegronwall} and \eqref{eq:gronwall} gives Lemma \ref{lem:interpol}, because
$\psi(N,h;\gb,t=0,1+M) = \frac{1}{2N} \log \frac{\cA'_N}{\cB_N}$,
where 
\begin{align}
\cA'_N &:= \E^{\otimes 2}\bigg[ e^{2(M+1)\gb^2 \sum_{n,m=1}^N \rho_{nm} \delta_n \delta'_m + h \sum_{n=1}^N (\delta_n + \delta'_n) + \frac12 \gb^2 \sum_{n,m=1}^N \rho_{nm} (\delta_n \delta_m + \delta'_n \delta'_m)} \delta_N \delta'_N \bigg]\\
\cB_N &:= \E^{\otimes 2}\bigg[ e^{h \sum_{n=1}^N (\delta_n + \delta'_n) + \frac12 \gb^2 \sum_{n,m=1}^N \rho_{nm} (\delta_n \delta_m + \delta'_n \delta'_m)}\delta_N \delta'_N \bigg].
\end{align}

\begin{proof}[Proof of Equations (\ref{eq:interpolation}) and (\ref{eq:equ_partiald})]
We are now left with computing $\partiald{t}\psi(N,h;\gb,t,\kappa)$ and $\partiald{\kappa}\psi(N,h;\gb,t,\kappa)$. Recall the definition \eqref{eq:defEngdh} and \eqref{eq:defHngbtkappa}, to get
\begin{multline}
\partiald{t}\psi(N,h,\gb;t,\kappa) = \frac{\gb}{4\sqrt{t}N} \sum_{n=1}^N \bbE \left[ \frac{\E_{N,\beta,h}^{\otimes 2}[\omega_n(\delta_n + \delta'_n)\exp\{H^{\otimes 2}_N(\gb,t,\kappa,\omega)\}]}{\E_{N,\beta,h}^{\otimes 2}[\exp\{H^{\otimes 2}_N(\gb,t,\kappa,\omega)\}]}\right]\\
 - \frac{\gb^2}{4N}\sum_{n,m=1}^{N} \rho_{nm} \bbE \left[ \frac{\E_{N,\beta,h}^{\otimes 2}[(\delta_n\delta_m + \delta'_n\delta'_m)\exp\{H^{\otimes 2}_N(\gb,t,\kappa,\omega)\}]}{\E_{N,\beta,h}^{\otimes 2}[\exp\{H^{\otimes 2}_N(\gb,t,\kappa,\omega)\}]}\right].
\end{multline}
By the Gaussian integration by part formula, we have
\begin{multline}
\bbE \left[ \frac{\omega_n \exp\{H^{\otimes 2}_N(\gb,t,\kappa,\omega)\}}{\E_{N,\gb,h}^{\otimes 2}[\exp\{H^{\otimes 2}_N(\gb,t,\kappa,\omega)\}]}\right] =  \sum_{ m=1}^{N}  \rho_{nm} \bbE \left[ \partiald{\omega_m}\left(\frac{\exp\{H^{\otimes 2}_N(\gb,t,\kappa,\omega)\}}{\E_{N,\gb,h}^{\otimes 2}[\exp\{H^{\otimes 2}_N(\gb,t,\kappa,\omega)\}]} \right) \right]\\
= \sqrt{t}\beta  \sum_{m=1}^N \rho_{nm} \bbE \left[ \frac{\sqrt{t}\gb(\delta_m +\delta'_m)\exp\{H^{\otimes 2}_N(\gb,t,\kappa,\omega)\}}{\E_{N,\gb,h}^{\otimes 2}[\exp\{H^{\otimes 2}_N(\gb,t,\kappa,\omega)\}]}\right. \hspace{5cm}\\
 \left. - \frac{\sqrt{t}\gb \exp\{H^{\otimes 2}_N(\gb,t,\kappa,\omega)\}\E_{N,\gb,h}^{\otimes 2}[(\delta_m +\delta'_m)\exp\{H^{\otimes 2}_N(\gb,t,\kappa,\omega)\}]}{\E_{N,\gb,h}^{\otimes 2}[\exp\{H^{\otimes 2}_N(\gb,t,\kappa,\omega)\}]^2} \right].
\end{multline}
Therefore, using the notation
\begin{equation}
\langle \cdot \rangle = \langle\cdot \rangle_{N,h,\omega}^{\gb,t,\kappa} := \frac{\E^{\otimes 2}_{N,\gb,h}[\, \cdot \, \exp\{H^{\otimes 2}_N(\gb,t,\kappa,\omega)\}]}{\E^{\otimes 2}_{N,\gb,h}[\exp\{H^{\otimes 2}_N(\gb,t,\kappa,\omega)\}]} \ ,
\end{equation}
we end up with
\begin{multline}
\label{eq:derivt}
\partiald{t} \psi(N,h,\gb;t,\kappa) = \frac{\gb^2}{4N} \sum_{ n,m=1}^{ N} \rho_{nm} \bbE[ \langle(\delta_n + \delta'_n)(\delta_m + \delta'_m)\rangle -\langle\delta_n + \delta'_n \rangle\langle \delta_m + \delta'_m \rangle]\\
 \hspace{3cm}- \frac{\gb^2}{4N} \sum_{ n,m=1}^{ N} \rho_{nm} \bbE[\langle\delta_n \delta_m \rangle + \langle\delta'_n \delta'_m \rangle]\\
 = \frac{\gb^2}{2N} \sum_{ n,m=1}^{ N} \rho_{nm} \bbE[\langle\delta_n \delta'_m \rangle] - \frac{\gb^2}{N} \sum_{ n,m=1}^{ N} \rho_{nm}  \bbE[\langle\delta_n \rangle \langle\delta_m\rangle],
\end{multline}
where we used that $\langle \gd_n\rangle =\langle \gd'_n\rangle$ (by symmetry in $\gd$ and $\gd'$) to simplify the last sum.
Similarly (and more easily, no Gaussian integration by part being needed), one gets
\begin{equation}
\label{eq:derivkappa}
\partiald{\kappa} \psi(N,h,\gb;t,\kappa) = \frac{\gb^2}{2N}\sum_{n,m=1}^N  \rho_{nm} \bbE[\langle\delta_n \rangle \langle \delta_m \rangle].
\end{equation}
From \eqref{eq:derivt}-\eqref{eq:derivkappa}, one gets \eqref{eq:interpolation} since $\sum_{n,m=1}^N \rho_{nm} \langle\delta_n \rangle \langle\delta_m\rangle \geq 0$ ($\gU$ is positive semi-definite).
To obtain \eqref{eq:equ_partiald}, one simply realizes that when $\kappa =0$, $\langle \gd_n\gd'_m\rangle = \langle \gd_n \rangle \langle \gd'_m\rangle = \langle \gd_n \rangle \langle \gd_m\rangle$.
\end{proof}
\end{appendix}

\bigskip

\noindent
{\bf Acknowledgments.} 
\normalsize
QB acknowledges support by a travel grant from the Simons Foundation.
JP acknowledges the support of ERC Advanced Grant 267356 VARIS and ANR project MEMEMO2 10–BLAN–0125–03. Part of this work was carried out during the YEP conference and the following workshop on Random Polymers at EURANDOM in January 2013, and the authors want to thank TU Eindhoven for its hospitality.


\bigskip



\begin{thebibliography}{AA}
 
 \bibitem{cf:A2}
Alexander, K.S.
\textit{The effect of disorder on polymer depinning transitions},
Commun.\
Math.\ Phys.\ , 279 (2008), 117--146.


\bibitem{cf:AZ}
Alexander, K.S., Zygouras, N.
\textit{Quenched and annealed critical points in polymer pinning models},
Comm.\ Math.\ Phys.\ 291 (2010), 659--689.

\bibitem{cf:AZ2}
Alexander, K.S., Zygouras, N.
\textit{Equality of Critical Points for Polymer Depinning Transitions with Loop Exponent One},
Ann. Appl. Prob \textbf{20} (2010), 356--366.

\bibitem{cf:Asm} 
 Asmussen, S. \textit{Applied probabilities and queues},
 Second Edition, Application of Mathematics 51, Springer-
Verlag, New-York (2003).

\bibitem{cf:B13}
Berger, Q.,
\textit{Comments on the Influence of Disorder for Pinning Model in Correlated Gaussian Environment}, ALEA,  Lat. Am. J. Probab. Math. Stat., {\bf 10} (2),  (2013) 953–-977.

\bibitem{cf:B13bis}
Berger, Q.,
\textit{Pinning model in random correlated environment: appearance of an infinite disorder regime}, J. Stat. Phys., Vol 155, Issue 3 (2014), 544--570.

\bibitem{cf:BCPSZ}
Berger, Q., Caravenna, F., Poisat, J., Sun, R. and Zygouras, N.,
\textit{The critical curves of the random pinning and copolymer models at weak coupling}, Commun. Math. Phys.,  Vol 326, Issue 2 (2014) 507--530

\bibitem{cf:Bingham}
N. H. Bingham, N. H.,  Goldie, C. M., and Teugels, J. L.,
\textit{Regular Variations},Cambridge University Press,
Cambridge (1987)


\bibitem{cf:BG}
T. Bodineau and  G. Giacomin, 
\textit{On the localization transition  of random copolymers near selective interfaces},
J. Statist. Phys. {\bf 117} (2004), 801--818.

\bibitem{cf:BGLT}
Bodineau, T., Giacomin, G., Lacoin, H., Toninelli, F.L.
\textit{Copolymers at selective interfaces: new bounds on the phase diagram},
J.\ Stat.\ Phys.\ 132 (2008), 603-–626.

\bibitem{cf:BdH}
Bolthausen, E., den Hollander, F.
\textit{Localization transition for a polymer near an interface},
Ann.\ Probab.\ 25 (1997), 1334--1366.

\bibitem{cf:BdHO}
Bolthausen, E., den Hollander, F., Opoku, A.A.
\textit{A copolymer near a selective interface: variational characterization of the free energy},
arXiv:1110.1315 (2011).


\bibitem{cf:CG}
Caravenna F., Giacomin G.
\textit{The weak coupling limit of disordered copolymer models},
Ann.\ Probab.\ 38 (2010), 2322--2378.

\bibitem{cf:CGG}
Caravenna, F., Giacomin, G., Gubinelli, M.
\textit{Large scale behavior of semiflexible heteropolymers},
Ann.\ Inst.\ H.\ Poincar\'e Probab.\ Stat.\ 46 (2010), 97--118.


\bibitem{cf:CSZ14}
Caravenna, F., Sun, R., Zygouras, N.
\textit{Polynomial chaos and scaling limits of disordered systems}, preprint, arXiv:1312.3357.

\bibitem{cf:CdH}
Cheliotis, D., den Hollander, F.
\textit{Variational characterization of the critical curve for pinning of random polymers},
Ann. Probab. 41 (2013), 1767--1805.

\bibitem{cf:Cornfeld}
Cornfeld, I. P., Fomin, S. V., Sinai, I. G.
\textit{Ergodic Theory},
Springer, New-York (1982)


\bibitem{cf:DGLT}
Derrida, B., Giacomin, G., Lacoin, H., Toninelli, F.L.
\textit{Fractional moment bounds and disorder relevance for pinning models},
Comm.\ Math.\ Phys.\ 287 (2009), 867–-887.




\bibitem{cf:G1}
Giacomin, G.
\textit{Random polymer models},
Imperial College Press (2007).

\bibitem{cf:G2}
Giacomin, G.
\textit{Disorder and critical phenomena through basic probability models},
Lecture notes from the 40th Probability Summer School held in Saint-Flour, 2010.
Springer (2011).

\bibitem{cf:GLT1}
Giacomin, G., Lacoin, H., Toninelli, F.L.
\textit{Disorder relevance at marginality and critical point shift},
Ann.\ Inst.\ H.\ Poincar\'e Probab.\ Stat.\ 47 (2011), 148--175.

\bibitem{cf:GLT2}
Giacomin, G., Lacoin, H., Toninelli, F.L.
\textit{Marginal relevance of disorder for pinning models},
Comm.\ Pure Appl.\ Math.\ 63 (2011), 233--2650.

\bibitem{cf:GT}
Giacomin, G., Toninelli, F.L.
\textit{Smoothing effect of quenched disorder on polymer depinning transitions},
Commun.\ Math.\ Phys.\ 266 (2006), 1--16.

\bibitem{cf:GT2}
Giacomin, G., Toninelli, F.L.
\textit{The localized phase of disordered copolymers with adsorption},
ALEA Lat. Am. J. Probab. Math. Stat.\ 1 (2006), 149--180.

\bibitem{cf:Szego}
Grenander, U.,  Szeg\"o, G.
\textit{Toepltiz forms and their applications},
California Monographs in Mathematical Science, University of California Press, (1958).

\bibitem{cf:H}
Hammersley, J. M.,
\textit{Generalization of the fundamental theorem on sub-additive functions},
Proc. Cambridge Philos. Soc.\ 58 (1962) 235--238.

\bibitem{cf:Harris}
Harris, A.B.
\textit{Effect of random defects on the critical behaviour of Ising models},
J. Phys. C. \textbf{7} (1974), 1671--1692.

\bibitem{cf:dH}
den Hollander, F.
\textit{Random polymers}.
Lectures from the 37th Probability Summer School held in Saint-Flour, 2007.
Springer-Verlag, Berlin (2009).

\bibitem{cf:L}
Lacoin, H.
\textit{The martingale approach to disorder irrelevance for pinning models}, 
Electron. Comm. Probab. \textbf{15} (2010), 418--427.

\bibitem{cf:M}
Monthus, C.
\textit{On the localization of random heteropolymers at the interface between two selective solvents},
Eur.\ Phys.\ J.\ B 13 (2000), 111–-130.



\bibitem{cf:P1}
Poisat, J.
\textit{Ruelle-Perron-Frobenius operator approach to the annealed pinning model with Gaussian long-range correlated disorder}.
Markov Processes and Related Fields, Vol. 19, no. 3 (2013)


\bibitem{cf:Polya}
P\'olya, G.
\textit{Remarks on characteristic functions}
Proceedings of the [First] Berkeley Symposium in Mathematical Statistics and Probability, 115--123, University of California Press (1949).


\bibitem{cf:T2}
Toninelli, F.L.
\textit{A replica-coupling approach to disordered pinning models}, Commun.\
Math.\ Phys.\ , 280 (2008), 389--401

\bibitem{cf:T08}
Toninelli, F.L.
\textit{Disordered pinning models and copolymers: Beyond annealed bounds}, Ann.\ Appl.\ Probab.\ , 18 (2008), 1569--1587

\bibitem{cf:T1}
Toninelli, F.L.
\textit{Coarse graining, fractional moments and the critical slope of random copolymers},
Electron.\ J. Probab.\ 14 (2009), 531--547.
 \end{thebibliography}
\end{document}